\documentclass[10pt,reqno]{amsart}
\usepackage{amsmath,amsfonts,amssymb,amscd,amsthm,amsbsy,bbm, epsf,calc,enumerate,color,graphicx,verbatim,cite,epstopdf}
\usepackage{color}
\usepackage{datetime}
\usepackage{soul}

\newcounter{hours}\newcounter{minutes}

\makeatletter
\newtheorem*{rep@theorem}{\rep@title}
\newcommand{\newreptheorem}[2]{%
\newenvironment{rep#1}[1]{%
 \def\rep@title{#2 \ref{##1}}%
 \begin{rep@theorem}}%
 {\end{rep@theorem}}}
\makeatother

\theoremstyle{theorem}

\newtheorem{thm}{Theorem}[section]
\newreptheorem{thm}{Theorem}

\newtheorem{lem}[thm]{Lemma}
\newtheorem{mainthm}{Theorem}

\newreptheorem{lem}{Lemma}

\theoremstyle{definition}

\theoremstyle{remark}                  

\def\E{{\mathbb E}}

\def\P{{\mathbb P}}
\def\R{{\mathbb R}}

\def\real{{\mathbb R}}

\def\bf 1{{\text{\raisebox{2pt}{$\chi$}}}}
\def\integer{{\mathbb Z}}

\def\indicator{{\mathbf 1}}
\def\ep{\varepsilon}
\def\e{\varepsilon}

\def\tr{\textnormal{tr}}

\DeclareMathOperator*{\osc}{osc}

\DeclareMathOperator*{\argmin}{\arg\!\min}
\def\argmin{\mathop{\arg\,\min}\limits}%

\numberwithin{equation}{section}

\begin{document}
\title{Quantitative homogenization of elliptic pde with random oscillatory boundary data}
\author{William M. Feldman, Inwon Kim}
\address{Department of Mathematics, UCLA, Los Angeles, CA 90024, USA}
\email{wfeldman10@math.ucla.edu, ikim@math.ucla.edu}
\thanks{W. M. Feldman and I. C. Kim both have been partially supported by the NSF grant DMS-1300445}
\author{Panagiotis E.  Souganidis} 
\thanks{P. E. Souganidis was partially supported by the NSF grants DMS-0901802 and DMS-1266383.}
\address{Department of Mathematics, The University of Chicago, Chicago, IL 60637, USA}
\email{souganidis@math.uchicago.edu}

\begin{abstract}
We study the averaging behavior of nonlinear uniformly elliptic partial differential equations  with random Dirichlet or Neumann boundary data oscillating on a small scale. Under conditions on  the operator, the data  and  the random media leading to concentration of measure, we prove an almost sure and local uniform  homogenization result with a rate of convergence in probability.

\end{abstract}

\maketitle
\section{Introduction}

 In this article we investigate the averaging behavior of the solutions to nonlinear uniformly elliptic partial differential equations  with random Dirichlet or Neumann boundary data oscillating on a small scale. Under conditions on the operator, the data  and the random media leading to  concentration of measure, we prove an almost sure and local uniform homogenization result with a rate of convergence in probability. 

\medskip

 In particular,  we consider the Dirichlet and Neumann boundary value problems 
\begin{equation}\label{main0}
\left\{
\begin{array}{lll}
F(D^2u^\e) = 0 &\hbox{ in }& U,\\[1.mm]
u^\e
= g(\cdot, \frac{\cdot}{\e}, \omega) &\hbox{ on }& \partial U,
\end{array}
\right.
\end{equation}
and 
\begin{equation}\label{main00:N}
\left\{
\begin{array}{lll}
F(D^2u^\e) = 0 &\hbox{ in }& U\setminus K,\\[1mm]
\partial_{\nu} u^\e
=g(\cdot, \frac{\cdot}{\e}, \omega) &\hbox{ on }& \partial U,\\[1mm]
u^\e = f &\hbox{ on } & \partial K,
\end{array}
\right.
\end{equation}
where  $U$ is a smooth bounded domain in $\R^d$ with $d \geq 2$, $K$ is a compact subset of $U$, $\nu$ is the inward normal, $F$ is positively homogeneous of degree one and uniformly elliptic, $f$ is continuous on $K$ and
$g=g(x,y,\omega)$ 
is bounded and Lipshitz continuous in $x,y$ uniformly in  $\omega$ belonging to a probability space $(\Omega,\mathcal{F},\P)$, and, for each fixed $x\in U$, stationary with respect to the translation action of $\real^d$ on $\Omega$ and strongly mixing  with respect to $(y,\omega)$ (the precise assumptions are  given in Section~\ref{sec: prob setting}).  

 \medskip
 
   We show that there exist a deterministic continuous functions $\overline{g}_D, \overline g_N :\partial U \to \real$ such that, as $\e \to 0$, the solutions $u^\e=u^\e(\cdot,\omega)$ of \eqref{main0} and \eqref{main00:N} converge almost surely and locally uniformly in $ U$ (with a rate in probability) to the unique solution $\overline{u}$ of respectively 
\begin{equation}\label{main} 
\left\{
\begin{array}{lll}
F(D^2\overline{u}) = 0 &\hbox{ in }& U,\\[1.5mm]
\overline{u} = \overline{g}_D &\hbox{ on }& \partial U,
\end{array}
\right.
\end{equation}
and 
\begin{equation}\label{mainN} 
\left\{
\begin{array}{lll}
F(D^2\overline{u}) = 0 &\hbox{ in }& U\setminus K,\\[1.5mm]
\partial_\nu\overline{u} = \overline{g}_N &\hbox{ on }& \partial U,\\[1.5mm]
\overline u=f  &\hbox{ on } & \partial K.
\end{array}
\right.
\end{equation}
\smallskip

\noindent  The homogenized boundary condition $\overline{g}$ (here and in the rest of the paper we omit the subscript and always denote the homogenized boundary condition by $\overline g$) depend on $F$, $\nu$, $d$ and the random field $g$.  The rate of convergence depends on the regularity of $U$, the continuity and mixing properties of $g$, the dimension $d$, the ellipticity ratio of $F$ and, in the case of the Neumann problem,  the bounds of $f$.  

\medskip

 We discuss next heuristically what happens as $\e\to 0$ in the Dirichlet problem \eqref{main0}.   
 It follows from the up to the boundary continuity of the solutions to \eqref{main0} that, close to the boundary,  $u^\e$ typically has unit size oscillations over distances of order $\e$. Therefore any convergence to a deterministic limit must be occurring outside of some shrinking boundary layer, where the solution remains random and highly oscillatory even as $\e \to 0$.  In order to analyze the behavior of the $u^\e$ near a point $x_0 \in \partial U$ with inner normal $\nu$ we ``blow up'' $u^\e$ to scale $\e$, that is we consider
 $$v^\e(y,\omega) = u^\e(x_0+\e y,\omega).$$
If homogenization holds,  then $v^\e(R\nu,\omega)$ should converge to $\overline{g}(x_0)$ for $R>0$ sufficiently large to escape the boundary layer. Noting that the random function $v^\e(\cdot,\omega)$ is uniformly continuous, as $\e\to 0$, we can approximate $v^\e(\cdot,\omega)$ by the solution of the half-space problem, obtained after  ``blowing up" in the tangent half-space at $x_0$,
\begin{equation}\label{eqn: cell}
\left\{
\begin{array}{lll}
F(D^2v) = 0  & \text{ in } & \{ y\in \real^d: y\cdot\nu > 0\}\\[1.5mm]
v(\cdot,\omega) = \psi(\cdot,\omega)  & \text{ on } & \{ y\in \real^d: y\cdot\nu=0\}.
\end{array}
\right.
\end{equation}
Formally, based on the problem satisfied by $v^\e(\cdot,\omega)$, we expect that,
\begin{equation}\label{eqn: ve to v}
 |v^\e(y,\omega) - v(y,\omega)| \to 0 \ \hbox{ as } \ \e \to 0 \ \hbox{ when } \ \psi(y,\omega) = g(x_0,y+\e^{-1}x_0,\omega).
 \end{equation} 
With the expectation that $u^\e(x+\e R\nu,\omega)$ can be approximated, for small $\e$ and large $R$, by $v(R\nu,\omega)$, we are led to consider, for  random fields $\psi$ satisfying assumptions similar to $g$ for fixed $x$,  the existence of an almost sure limit, as $y \cdot \nu \to \infty$, of $v(y,\omega)$. This is the analogue, in our setting, of the cell problem in classical homogenization theory.  
\medskip

We say that the cell problem \eqref{eqn: cell} has a solution, if there exists a constant $\mu = \mu(\nu, F, \psi)$, often referred  to as the ergodic constant, such that 
\begin{equation}\label{convergence}
 \lim_{R \to \infty} v(R\nu,\omega) = \mu \ \hbox{ almost surely.}
  \end{equation}
  As indicated above there are two main steps in the argument.  The first is to show the existence of the limit \eqref{convergence} for the cell problem.  This is where all the assumptions on $F$  and the randomness come in (see Section~\ref{sec: prob setting}).  The second is to show that the approximation of $u^\e$ near the boundary by the cell problem \eqref{eqn: cell} holds with quantitative estimates so that the convergence results for the cell problem can be used to identify the boundary condition for the general domain problem.
  
  \medskip
  
 Similar heuristic arguments lead to the Neumann cell problem, which is to show that there exists a deterministic (ergodic) constant $\mu = \mu(\nu,F,\psi)$ such that, if $v_R (\cdot,\omega)= v_{R,\nu}$ is the unique bounded solution of
   \begin{equation}\label{eqn: neumann cell0}
\left\{
\begin{array}{lll}
F(D^2 v_R) = 0 &\hbox{ in } &\{ y \in \real^d: 0<y\cdot\nu <2R\},\\[1.5mm]
\partial_{\nu}v_R(\cdot,\omega) = \psi(\cdot,\omega) &\hbox{ on } & \{y \in \real^d: y\cdot\nu = 0\},\\[1.5mm]
v_R(\cdot,\omega) = 0&\hbox{ on } &\{y \in \real^d: y\cdot\nu = 2R\},
\end{array}
\right. 
\end{equation}
then
\begin{equation}\label{cell:N0}
\lim_{R\to\infty} \dfrac{v_R(R\nu,\omega)}{R} = \mu \ \hbox{ almost surely.}
\end{equation}
\smallskip

 In the Neumann problem,  it is the gradient $Du^\ep$ of the solution to \eqref{main00:N} that has an oscillatory boundary layer,  outside of which it should be approaching a deterministic constant far from the Neumann part of the boundary.

 \medskip
 
 In order to have any hope for the convergence described in \eqref{convergence} and \eqref{cell:N0}, it is necessary to impose  assumptions on the randomness.  To motivate them, we consider the Dirichlet cell problem \eqref{eqn: cell} in the linear case where we can represent $v(R\nu,\omega)$, using the Poisson kernel, as
 $$ v(R\nu,\omega) = \int_{\{y \in \real^d: y\cdot\nu = 0\}} P(Re_d,y) \psi(y,\omega) dy.$$
 Since $P(R\nu,y) \sim R^{1-d}$ on $B_R \cap \{y \in \real^d: y\cdot \nu = 0\}$, $v(R\nu,\omega)$ is essentially the average of the boundary values on $B_R \cap \{ y \in \real^d: y\cdot \nu = 0\}$.  In this case the convergence, as $R\to \infty$, is a consequence of the ergodic theorem as long as the boundary data satisfies two important assumptions, namely stationarity and ergodicity, which we describe next.
 
 \medskip
 
 Firstly the distribution of $\psi(y,\omega)$ should not depend on $y$, in other words $\psi(y,\omega)$ must be stationary with respect to translations parallel to the hyperplane $\{y \in \real^d: y\cdot \nu = 0\}$.  On the other hand, in view of   \eqref{eqn: ve to v},  we actually need to consider the cell problem for translations of the boundary data parallel to $\nu$ as well and so we actually require stationarity of $\psi$ with respect to all $\real^d$ translations.  Without some assumption of this form one can easily construct examples for which the limit will not exist. 
 
 \medskip
 
For reference it is useful to consider the periodic version of our problem.  In this case, there is a $\integer^d$ translation action under which the underlying probability space, which is the torus $\mathbb{T}^d$, is stationary and ergodic.  When the normal direction $\nu$ is rational, the boundary values are not stationary with respect to the translations parallel to $\nu$.  
 
 \medskip
 
 One possible way to address this problem (see Section \ref{prelim} for more discussion of this version of the problem) is to consider data defined on a hyperplane which is stationary with respect to an $\real^{d-1}$ or $\integer^{d-1}$ action, and then define data on pieces of the boundary of the general domain by ``lifting up from the hyperplane".  This is in contrast to the setting described already which we refer to as assigning boundary data by ``restricting from the whole space".  More specifically, one could take $\psi(y,\omega)$ on $\real^{d-1}$ and a diffeomorphism $\zeta$ from an open subset of $\real^{d-1}$ to an open subset of $\partial U$ and then define the boundary data, for the general domain problem, by
  \begin{equation}\label{eqn: prj up}
   g^\e(x,\omega) = \psi(\e^{-1}\zeta^{-1}(x),\omega).
   \end{equation}

   \medskip
 
 Secondly, we need to assume  that  the action of the translations of $\real^d$ on $\Omega$ is ergodic.  Indeed some assumption on the long range decorrolation of the values of $\psi(y,\omega)$, ergodicity being the weakest is always necessary in to prove a law of large numbers/ergodic theorem-type of result.  The exact form of the ergodic behavior actually turns out to be quite a delicate issue for boundary data homogenization because it does not necessarily restrict to lower dimensional subspaces.  
 
 \medskip
 
 An instructive way to understand this difficulty is again to consider the periodic version of the problem. In this case, the translations parallel to $\{y\in\real^d:   y \cdot \nu=0 \}$ are not ergodic when the normal direction is rational; for more discussion see Choi and Kim \cite{ChoiKim12} and Feldman \cite{Feldman13}.  It turns out, however,  that it is enough that most (in an appropriate sense) directions yield an ergodic action.  
 
 \medskip
 
 It is not at all clear to the authors what kind of generalization of the periodicity assumption would yield this kind of homogenization for almost every direction.  Again if we take boundary data on the general domain by ``lifting up from hyperplanes" as in \eqref{eqn: prj up}, the concern above 
 is not an issue.  On the other hand, if one assumes a more quantitative decay of correlations like strong mixing, the translation action of the $d-1$ dimensional subspaces $\partial P_{\nu}$ on $\Omega$ will be ergodic and the resulting rate of homogenization is uniform in the normal direction.

\medskip

We also remark that, as is always the case in random homogenization, we lack the compactness of the periodic setting.  All the  major assumptions are used to overcome the above difficulties, that is   stationarity and ergodicity on hyperplanes and lack of compactness.

\medskip

Qualitative homogenization results in the stochastic setting for elliptic equations with oscillations in the interior of the domain typically  rely on identifying a quantity which controls, via Alexandrov-Bakelman-Pucci (ABP for short)-type inequalities, the asymptotic behavior of the solution and whose ergodic properties can be studied using the sub-additive ergodic theorem;  see Caffarelli, Souganidis and Wang \cite{CSW05} and, later, Armstrong and Smart \cite{CS10}.  
\medskip

It is not clear to the authors whether such a quantity exists for boundary homogenization.  This is related to the fact that it is not known whether there exists an estimate analogous to the ABP-inequality  controlling solutions to boundary value problems for linear elliptic equations with bounded measurable coefficients in terms of a measure theoretic norm of the boundary data. Relatedly it is known that the harmonic measure for  linear  non-divergence form operators with only bounded measurable coefficients may be singular with respect to the surface measure on the boundary and, in fact, may be  supported on a set of lower Hausdorff dimension (see Wu \cite{Wu96} or Caffarelli, Fabes and Kenig \cite{CFK81} for the divergence form case.)

\medskip

On the other hand, an argument, which is more in the spirit of the linear problem, works well here since it turns out that the value of the solutions of the cell problems can still be thought of as Lipschitz ``nonlinear averages'' of the boundary values.  This observation lends itself to using tools from the theory of concentration of measure,  which,  generally speaking, provide estimates in probability for the concentration  of Lipschitz functions of many independent random variables about their mean. There is more detailed discussion about this later in the paper.

 \medskip

 To give an idea of the type of results we obtain, we state informally the main theorems without any technical assumptions.  Exact statements are  given in the next section.  In the case that $F$ is either convex or concave, there is a very powerful concentration inequality due to Talagrand which allows the homogenization result to hold without any additional assumptions on $F$.  Without convexity/concavity,  the available concentration inequalities either lead to a restriction on the ellipticity ratio of $F$ or depend on stronger assumptions on the random media.  We state the results separately depending on this property of $F$.

\begin{mainthm}\label{thm:main0}
Let $F$ be positively homogeneous of degree one, uniformly elliptic and either convex or concave. Assume that the stationary random field $\psi$ is bounded, uniformly H\"{o}lder continuous and satisfies a strong mixing condition with sufficient decay.  Then the Dirichlet and Neumann cell problems \eqref{eqn: cell} and \eqref{eqn: neumann cell0} homogenize. 
\end{mainthm}

\begin{mainthm}\label{thm:main01}
Let $F$ be positively homogeneous of degree one and uniformly elliptic and $\psi$ be a bounded and uniformly H\"{o}lder continuous stationary random field.\\
(i) Under a restriction on the ellipticity ratio of $F$ and  if $\psi$ satisfies a strong mixing condition with sufficient decay,  the  cell problems \eqref{eqn: cell} and \eqref{eqn: neumann cell0} homogenize.\\
(ii) If the random media is of ``random checkboard"-type, stationary with respect to the $\integer^{d-1}$ translation action on $\{ y \in \real^d: y \cdot \nu = 0\}$, and has  a log-Sobolev inequality, then the cell problems \eqref{eqn: cell} and \eqref{eqn: neumann cell0} homogenize without any assumption on the ellipticity ratio of $F$.
\end{mainthm}

\begin{mainthm}\label{thm:main02}
Under the assumptions of either Theorem~\ref{thm:main0} or Theorem~\ref{thm:main01} part (i),  the general domain Dirichlet and Neumann boundary value problems \eqref{main0} and \eqref{main00:N} homogenize to  \eqref{main} and \eqref{mainN}, with the homogenized boundary condition determined by the corresponding cell problems \eqref{eqn: cell} and \eqref{eqn: neumann cell0}. 

\end{mainthm}

See Section~\ref{prelim} for more precise results.  We expect that  methods very similar to the ones presented here will also show that the homogenization of the cell problems \eqref{eqn: cell} and \eqref{eqn: neumann cell0} for boundary data as, for example in Theorem~\ref{thm:main01} part~(ii), will yield a corresponding version of Theorem \ref{thm:main02} for general domains with boundary data defined by ``lifting up from the hyperplane" as in \eqref{eqn: prj up} (see Section~\ref{prelim} for more details).

\medskip

 Our approach can be extended to equations with nonhomogeneous of degree one nonlinearities and  spatial oscillations inside the domain as long as the latter do not ``interact'' with the second-order derivatives, because the lower order terms scale out during the ``blow up" procedure that leads to the cell problems. 
 
 \medskip
 
 Understanding the homogenization in the presence  of random oscillations in both the second order term and the boundary data is still open.  The main obstacle appears to be the fact that the methods applied to study the homogenization in the interior and on the boundary do not provide sufficient control to deal with the additional boundary layer created by the combined oscillations.

\medskip

We present next a short review of the existing literature. Since there is a large body of work concerning the homogenization of elliptic pde, we organize this review  around oscillatory and non-oscillatory boundary value problems.  In each case we  discuss general qualitative homogenization results and quantitative error estimates. 

\medskip

Classical references for homogenization of linear (divergence and non-divergence form) operators in periodic media are the books Benssousan, Lions, Papanicolaou \cite{BPL78} and Jikov, Koslov, Ole{\u\i}nik \cite{JKO94}, while for fully nonlinear problems we refer to Evans \cite{Evans89} and Caffarelli  \cite{C99}.  The most general nonlinear result, without any convexity assumptions on $F$, is due to Caffarelli and Souganidis \cite{CS10} using $\delta$ viscosity solutions. 

\medskip

The first results in random media for linear divergence and non-divergence form operators are due to Papanicolaou and Varadhan \cite{PV79, PV82} and Kozlov \cite{kozlov}.  Nonlinear variational problems were studied by Dal Maso and Modica \cite{DMM86}.  The first general nonlinear result was obtained by Caffarelli, Souganidis and Wang \cite{CSW05}, who introduced a sub-additive quantity, based  on a family of auxiliary obstacle problems, which controls the behavior of the solutions using the ABP-inequality  and, in view  of the sub-additive ergodic theorem, has an ergodic behavior.  Error estimates for linear divergence and non-divergence form equations under Cordes-type assumptions were obtained by Yurinski{\u\i} \cite{Y85,Y88}, while Gloria and Otto \cite{GO11}, Gloria, Neukamm and Otto \cite{GNO} and Marahrens and Otto \cite{MO}  proved optimal rates of convergence in the discrete setting  for linear divergence form elliptic pde.  For nonlinear non-divergence form equations the first error estimate (logarithmic rate) was obtained by \cite{CS10} in the strong mixing setting;  under some assumptions this was upgraded recently to an algebraic error by Armstrong and Smart \cite{AS13}.

\medskip

Much less is known about the homogenization of oscillatory boundary data. In the linear divergence form case with co-normal Neumann boundary data the homogenization is proved in the book \cite{BPL78}.  In the periodic setting, some special cases were discussed in Arisawa\cite{Arisawa03}, the Neumann problem in a special half-space setting with periodic boundary data was studied  by Barles, Da Lio, Lions and Souganidis \cite{BDLS08} (some results were obtained earlier by Tanaka \cite{T84} using probabilistic methods), and recently, for general domains and  rotationally invariant equations,  by Choi and  Kim \cite{ChoiKim12} and Choi, Kim and Lee \cite{CKL12}. Qualitative results for the oscillatory Dirichlet problems in periodic media were obtained  by Barles and Mironescu \cite{BarlesMironescu12} in the half-space setting and, recently, by Feldman \cite{Feldman13} in general domains.   In the periodic case, the analysis for   general domains requires either a careful geometric analysis, such as in \cite{ChoiKim12}, to obtain a uniform modulus of continuity,  or an argument as in \cite{Feldman13} that ignores discontinuities of the data in small parts of the boundary.  We also point out the recent results of Garet-Varet and Masmoudi \cite{GVM11, GVM12} about systems of divergence-form operators with oscillatory Dirichlet data in periodic media with error estimates. In a similar vein are the results of Kenig, Lin and Shen \cite{KLS12} on the rate of convergence for interior homogenization with Dirichlet or Neumann boundary conditions in Lipschitz domains. 

\medskip

As far as we know, the results of this paper are the first concerning homogenization of nonlinear, and, for general domains, even of linear  Dirichlet and Neumann problems with random oscillatory data.  Our approach is based on measuring,  using appropriate barriers, the fluctuation of the interior values of the solutions to the cell problems in terms of the randomness on the boundary.  Once such control has been established, we use tools from the theory of concentration inequalities to prove estimates on the deviation of the interior values of the solutions from their mean.  Estimates from the elliptic theory and the maximum principle are then used to show that the means converge to a deterministic constant.  The different versions of our results in the convex and non-convex cases are due to the nature of the available concentration estimates.  Once the homogenization of the cell problem for half spaces in any direction has been established, we employ  estimates from the elliptic theory to make rigorous the ``blow up" argument described earlier to show that homogenization occurs for the general domain problem. We also discuss separately the cell problems of, what we called above, the ``lifting up from the hyperplane'' problem. The arguments in this case are similar but, since the  probabilistic setting is simpler, we are able to obtain general results as far as $F$ is concerned.

\medskip

The paper is organized as follows:  In Section \ref{prelim} we introduce the notation and some conventions, the assumptions on the equations,  the probabilistic setting, the concentration inequalities we use in the paper and some useful technical facts about uniformly elliptic pde.  Then we present  the precise statements of the results.  In Section \ref{cell} we begin with an outline of the proof of the homogenization of the Dirichlet cell problem, then we continue with the details which consist of two steps.  The first is the Lipschitz estimates of solutions of the cell problem  in terms of the boundary data and the second is the probabilistic arguments based on the aforementioned concentration results.  The  presentation  is divided into two parts depending on the assumptions on $F$ and the randomness.  We also prove the  continuity of the homogenized boundary condition and the homogenization of the cell problem in the  ``lifting up from the hyperplane'' setting.  Section \ref{general} contains the proof of homogenization for the Dirichlet problem in a general domain.  The Neumann problem is discussed in Section \ref{neumann}.  The structure of the proof mirrors that of the Dirichlet problem.

\subsection*{Acknowledgements}  The first author would like to thank Nicholas Cook for many helpful discussions related to the subject of concentration inequalities.

\section{Preliminaries, Assumptions and Statements of Results}\label{prelim}
\subsection*{Notation and some terminology/conventions}

We denote by $\mathcal{M}^d$, $\tr M$ and $I_d$ the class of $d \times d$ symmetric matrices with real entries, the trace of $M\in \mathcal{M}^d$ and the the $d\times d$ identity matrix respectively. We write $Q$ for the unit cube $[0,1)^{d-1}$.  For each $\nu \in S^{d-1}$, the sphere in $\real^d$, $P_{\nu}:=\{x \in \real^d: x\cdot \nu>0\}$,  $\Pi_{\nu}:=\{x \in \real^d: 0<x\cdot \nu <1\}$ and for $r>0$ we call $\Pi_\nu^r = r \Pi_\nu$.  We call $B_r(x_0) = \{ x \in \real^d: |x-x_0| < r\}$ and $B_r$ and $B^+_r$ refer to $B_r(0)$ and $B_r(0)\cap P_{e_d}$ respectively. Given $\nu \in S^{d-1}$, we also write $x'=x-x\cdot \nu$, and, for  $L>0$,  we use the cylinders $\textup{Cyl}_{\nu,L} = \{x\in \real^d: |x'| \leq L\} \times \{ x\in \real^d: 0 \leq x \cdot \nu \leq L\}.$ Moreover, $2^{\mathbb{N}}$ is the set of dyadic numbers, ${\mathbf 1}_A$ is the indicator function of the set $A$ and $C^{0,\alpha}(D)$, $C^{0,1}(D)$, $C^{1,\alpha}(D)$ and $C^2(D)$ are the spaces of $\alpha$-H\"{o}lder continuous, Lipschitz continuous, continuously differentiable with $\alpha$-H\"{o}lder continuous derivatives and twice continuously differentiable functions $f:D\to \real$ with norms respectively  $\|f\|_{C^{0,\beta}(D)}$, $\|f\|_{C^{0,1}(D)}$, $\|f\|_{C^{1,\alpha}(D)}$ and $\|f\|_{C^{2}(D)}$.  We write  $\osc_D f = \sup_D f - \inf_D f$ for the oscillation over $D$ of the continuous function $f$  and  $\|f\|_{\infty, D}$ for the $L^\infty$ norm of the  bounded function  $f:D \to \real$; if there is no ambiguity for the domain, for simplicity, we only write $\|f\|_{\infty}.$ If $U$ is an open subset  of $\real ^d$, then, for $r>0$, $U_r:=\{ x\in U: \textup{dist}(x, \partial U) >r \},$ where $\textup{dist}$ is the usual Euclidean distance. 
For $a\in \real$,  $a_+$ and $a_-$ are respectively  the positive and negative parts of $a$. Given a metric space $X$, $\mathcal B(X)$ is the $\sigma$-algebra of the Borel sets of $X$ associated with the metric.  On product spaces, for example $\Xi = X^{\integer^n}$, the notation $\mathcal{B}(\Xi)$ will refer to the cylinder $\sigma$-algebra of Borel sets.  For bounded linear operators $L:\ell^2(\integer^n) \to  \ell^2(\integer^n)$ we write $|L|_{\ell^2}$ for the usual operator norm.  

\medskip
We say that constants are universal, if they only depend on the underlying parameters of the problem, which are the ellipticity ratio of $F$, the dimension $d$, the constants associated with the $C^2$-property of $U$ and the mixing conditions.  Also note that constants may change, without explicitly said,  from line to line. Given two quantities $A$ and $B$ we write
$ A \lesssim B \ \hbox{ when } \ A \leq CB \ \hbox{ for some universal } C.$  If $A \lesssim B$ and $B \lesssim A$, then we write $A \sim B$.

\medskip 

Throughout the paper subsolutions, supersolutions and solutions should be interpreted in the Crandall-Lions viscosity sense.  We refer to Crandall, Ishii and Lions \cite{CIL} for the facts we use throughout the paper about viscosity solutions.  In the special case of uniformly elliptic second order equations the book of Caffarelli and Cabr{\'e} \cite{CC95} is a useful reference.

\medskip

We work on a probability space $(\Omega,\mathcal{F},\P)$.  If $(\Xi,\mathcal{G})$ is a measurable space, $\xi : \Omega \to \Xi$ is a random variable, if it is a measurable mapping in the sense that $ \xi^{-1}(E) \in \mathcal{F}$ for all $E \in \mathcal{G}$.  When $\Xi = \real$ we write $\E \xi = \int \xi d\P$ when the integral is defined. If $f:D\times \Omega \to \real,$ we write $\E f(x)$ for $\int_\Omega f(x,\omega) d\P(\omega).$  If $\Xi$ has a metric space structure and $\mathcal{G} = \mathcal{B} (\Xi)$, then the composition of a continuous function $f: \Xi \to \real$ with a random variable $\xi$, $f\circ \xi : \Omega \to \real$, is still measurable mapping with respect to  $\mathcal{B} (\real)$.  This justifies the fact that the various functions on $\Omega$ we consider throughout the paper are indeed random variables.  Finally, $\sigma(O_a :a\in A)$ is the smallest $\sigma$ algebra containing the collection  $(O_a)_{a\in A} \subset \mathcal F$.

 \subsection*{The probabilistic setting and concentration inequalities}\label{sec: prob setting}
 We are given a probability space $(\Omega , \mathcal{F}, \P)$,  which is endowed with the action of the group of  $\real^d$-translations 
 $( \tau_y )_{y \in \real^d}$ such that, for  each $y \in \real^d$, $\tau_y : \Omega \to \Omega$ is measurable and measure preserving.  More specifically, for every measurable set $A \in \mathcal{F}$ and every $y,z \in \real^d$, 
 $$ \P(\tau_y(A))= \P(A) \ \hbox{ and } \ \tau_z\tau_y = \tau_{y+z}, \ \tau_0 = \textup{id}. $$
 We say that the action is ergodic if, for any $E \in \mathcal{F}$,
  $$ \hbox{ if } \ \tau_x E = E \ \hbox{ for all } \ x \in \real^d, \ \hbox{ then } \ \P(E) \in \{0,1\}.$$
 We will also use the above notions in the context of a $\integer^n$-action on a probability space and  the definitions are analogous.  
 
 \medskip
 
 Given a bounded subset  $O$ of $\real^d$ and a random field $f: \real^d \times \Omega \to \real$, $\mathcal{G}(O)$ is the $\sigma$-algebra  
  $$\mathcal{G}(O) : = \sigma (\{f(y,\cdot) : y \in O\}).$$
The  $\phi$-mixing rate function 
 \begin{equation}\label{mixing}
  \phi(r) : = \sup\{| \P(E | F) - \P(E)| : \textup{dist}(O,V)\geq r, \ E \in \mathcal{G}(O), \ F \in \mathcal{G}(V)  \ \hbox{with} \  \P(F) \neq 0\}
  \end{equation}
measures of the decorrelation of the values of the random field $f$.  If we want to emphasize the dependence on the random field $f$, we write $\phi_f$.

\medskip

We say that 
  \begin{equation}\label{phi mixing}
  \hbox{ the random field $f$ is  $\phi$-mixing if } \  \phi(r) \to 0 \ \hbox{ as } \ r \to \infty.  
  \end{equation}
 We will also use the $\phi$-mixing condition in the context of random fields on a $\integer^n$-lattice; the definition is completely analogous.  
 
  \medskip
  
 \noindent  Let $A$ be a complete separable metric space and $\mathcal{B}({A})$ the associated Borel $\sigma$-algebra. Fix $n \in \mathbb{N},$  let $m$ be a probability measure on $\Xi:= {A}^{\integer^{n}}$ with the cylinder $\sigma$-algebra and assume that the random field, $(X_j)_{j \in \integer^n}$, on $\integer^n$ given by the coordinate maps  satisfies the $\integer^n$ -lattice version of the $\phi$-mixing condition \eqref{phi mixing}. 
 \medskip
 
We say that  $f: \Xi \to \real$ is Lipschitz with respect to the $\alpha$-weighted Hamming distance with weight  $\alpha \in \ell^2(\integer^n)$ if
 \begin{equation}\label{alpha hamming}
  |f(X)-f(Y)| \leq \sum_{ i \in \integer^n} \alpha_i  \mathbf 1_{X_i \neq Y_i}. 
  \end{equation}
  
  \medskip
  
  We will be making use of the following concentration inequalities in the strong mixing setting due to Marton \cite{Marton03} and Samson \cite{Samson00}.  This first  is about functions which are Lipschitz with respect to the Hamming distance \eqref{alpha hamming}.  

  \begin{thm}\label{concentration1}
There exist  positive constants $c,C$ depending only on $n$ such that,  for any $f: \Xi \to \real$ which is $1$-Lipschitz with respect to the $\alpha$-weighted Hamming distance,
 $$ m(\{|f-\int f dm| > t\}) \leq C \exp \left(-\frac{ct^2}{|\alpha|_{\ell^2}(\sum_{i \in \integer^n}\phi(|i|)^{1/2})^2}\right).$$
 \end{thm}
 
 \medskip
\noindent  Suppose additionally that ${A}$ is a closed convex subspace of some separable Banach space with norm $\| \cdot \|$ and,   for $1\leq p < \infty$, define the $ \ell^p$-distances  on $\Xi^{\integer^n}$ by
$$ |X-Y|_{\ell^p} := (\sum_{j \in \integer^n} \|X_j-Y_j\|^p)^{1/p} \ \hbox{ for } \ X,Y \in \Xi^{\integer^n}.$$

\smallskip

 Concentration inequalities with respect to the Hamming, or $\ell^0$, distance on a product space (which is just $d_{\ell^0}(X,Y) = \sum \mathbf 1_{X_i \neq Y_i}$) are often suboptimal when the underlying space also has a Euclidean, that is  $\ell^2$,  distance.  This is essentially due to the fact that being $1$-Lipschitz with respect to $\ell^0$-distance on the Hamming cube $\{0,1\}^n$ is worse by a factor of $\sqrt{n}$ than being $1$-Lipschitz with respect to $\ell^2$-distance. For functions which are $1$-Lipschitz with respect to the  $\ell^2$-distance and convex, there is  the following result of Samson \cite{Samson00}, which is an extension of Talagrand's concentration inequality. 
 \begin{thm}\label{concentration2}
There exist positive constants  $c,C$ that depend only on $n$ such that, 
  for any convex and $1$-Lipschitz with respect to $\ell^2$-metric  $f: \Xi \to \real$, 
\begin{equation}\label{eqn: talagrand} 
m(\{|f-\int f dm| > t\}) \leq C \exp \left(-\frac{ct^2}{(\sum_{i \in \integer^n}\phi(|i|)^{1/2})^2}\right). 
\end{equation}
 \end{thm}
  
 \medskip
 
  Talagrand's inequality is very powerful but it requires convexity.  This assumption  appears to be, in general, necessary. However,  the known counterexamples without convexity do involve some irregularity of the underlying measure. 
  
 \medskip 
   A concentration inequality like \eqref{eqn: talagrand} does hold for Gaussian measures without the convexity assumption on the Lipschitz function $f$.  A less stringent assumption, which also yields concentration for general Lipschitz functions, is the so-called  log-Sobolev inequality that we describe next.
  
\medskip

A probability measure $m$ on $\mathcal B(\real^n)$ is said to  have  a  log-Sobolev inequality (LSI for short), if there exists $\rho>0$ so that, for any   locally continuously differentiable $f : \real^n \to \real$,   
 \begin{equation}
 \int f^2 \log \tfrac{f^2}{\int f^2 dm} dm \leq \frac{1}{2\rho} \int \sum_{j=1}^n |\partial_j f|^2 d m.
 \end{equation}

 The LSI can be defined on infinite product spaces as well, for example, $\real^{\integer^n}$. It turns out that measures which have log-Sobolev inequality satisfy the following  dimension independent concentration property; for a proof we refer to Ledoux  \cite{ledoux}.
 \begin{thm}\label{concentration0}
 Suppose that $m$ has LSI with constant $\rho$. For any   $1$-Lipschitz with respect to the $\ell^2$-metric  $f : \real^{n} \to \real$, 
  $$ m (\{ |f-\int f d m |>t \}) \leq 2\exp(-2\rho t^2).$$
 The same result holds on infinite product spaces like $\real^{\integer^n}$ as well.
 \end{thm}

We list next as lemmas two important properties of the log-Sobolev inequalities which are used  to construct large classes of measures with LSI. The first  is the so-called tensorization property, which says that LSI is preserved under taking product measures; we refer again to  \cite{ledoux} for more explanations.

 \begin{lem}\label{lem: tensorization}
Fix  $N \in \mathbb{N}$ and assume that,  for $1\leq j \leq N$,  the measures  $m_j$ on $\real^{n_j}$ have LSI with constants $\rho_j$.  Then the product measure $m = m_1 \otimes \cdots \otimes m_N$ on the product space $\real^{n_1+\cdots n_N}$ also has LSI with constant $\rho= \min \rho_j$.
 \end{lem}

The second is that the pushforward $\phi_\# m$ of a measure $m$ with LSI under a Lipschitz transformation $\phi$ still has a log-Sobolev inequality;
recall that $\phi_\# m(E) := m(\phi^{-1}E)$ for any $E \in \mathcal B(\real^n)$. 

\begin{lem}\label{pushforward}
Let $m$ be a measure on $\real^n$ with LSI and $\phi: \real^n \to \real^n$  a (globally) Lipschitz and locally continuously differentiable function. 
Then the pushforward measure  $\phi_{\#}m$ of  $m$ by $\phi$ also has log-Sobolev inequality.
\end{lem}
\begin{proof} 

Let  $f:\real^n \to \real$ be locally continuously differentiable. Writing, for notational convenience, $\tilde{f}: = f \circ \phi$, we find
 \begin{align*}
 \int f^2 \log \tfrac{f^2}{\int f^2 d\phi_\#m} d\phi_\#m &= \int \tilde{f}^2 \log \tfrac{\tilde{f}^2}{\int \tilde{f}^2 dm} d m  \leq \frac{1}{2\rho} \int |D\tilde{f}|^2 d m \\
& \leq \frac{\sup |D\phi|_{\ell^2}^2}{2\rho} \int |Df\circ \phi|^2 d m = \frac{\sup |D\phi|_{\ell^2}^2}{2\rho} \int |Df|^2 d \phi_{\#}m.
 \end{align*}
 \end{proof}
 
Combining Lemma~\ref{lem: tensorization}  and Lemma~\ref{pushforward}, it is possible to construct large classes of measures with LSI starting from some simple examples of such measures, like the uniform and Gaussian measures on the unit interval  and $\real$ respectively, which are the prototypical examples of measures with log-Sobolev inequality; see  \cite{ledoux}. 
 
 \medskip

For example, the distribution of $X  = (X_1,\dots,X_n) \in \real^n$, with  the $X_i$'s i.i.d. Gaussians with mean $0$ and variance $1$,  has LSI with the same constant $\rho$ as the distribution of the $X_i$'s.  Moreover, given a $n \times n$ matrix $\sigma$,  the random variable  $Y = \sigma X$  is also Gaussian with covariance matrix $\Sigma=(\E Y_iY_j)_{1\leq i,j \leq N}$ and its distribution  $m_{\sigma}$  on  $\real^n$ has LSI, since
 \begin{equation}
 \int f^2 \log \tfrac{f^2}{\int f^2 dm_\sigma} dm_\sigma \leq \frac{|\Sigma|_{\ell^2}^2}{2\rho} \int |Df|^2 d m_\sigma.
 \end{equation} 

\medskip
     
Based on the discussion above  we can construct now some random fields on $\integer^{d-1}$ with LSI.  
\medskip

Let  $X = (X_j)_{j \in \integer^{d-1}}$ be i.i.d. and such that $\textup{law}(X_0)$ has log-Sobolev inequality on $\real$.  Then, by Lemma~\ref{lem: tensorization},  the law, $m$, of $X$ on $\real^{\integer^{d-1}}$ has LSI. 
  
  \medskip

In the Gaussian case we can accommodate a more general mixing condition. Let $(Y_j)_{j \in \integer^{d-1}}$ be Gaussian random variables and  $ Y'= \phi(Y)$, where $\phi$ is Lipschitz with respect to $| \cdot |_{\ell^2}$ and $|\phi |_{\ell^\infty} \leq 1$.  If the covariance $\Sigma$  of  the  $Y_j$'s, with elements  $\Sigma_{ij} = \E Y_iY_j$,  has a finite operator norm as  a map  from ${\ell^2}$ to  $\ell^2$, then the law, $m$, of $Y'$ has LSI on $\real^{\integer^{d-1}}$. The finiteness of  $|\Sigma |_{\ell^2}$  follows from the summability in $\integer^{d-1}$ of the covariances,  since
$$ |\Sigma|_{\ell^2}^2  \leq  (\max_i \sum_j |\Sigma_{ij}| ).$$

\smallskip

We note that from the above example we may infer, at least heuristically, that the assumption of a log-Sobolev inequality amounts to some regularity of the $1$-dimensional distributions combined with a summable decay of correlations.

  \smallskip

\subsection*{The assumptions} We present next our assumptions. For  the convenience of the reader we divide them into separate groups.

 \smallskip
  
{\it The domain.} We assume that 
\begin{equation}\label{domain1}
 U \subset \real^d  \  \text{ has \  $C^2$-boundary},  
 \end{equation}
and 
\begin{equation}\label{K}
K \  \text{ is a compact subset of $U$}.
\end{equation}
The regularity of $\partial U$ means, in particular,  that, for each $x \in \partial \Omega$, there exists  an inward normal direction $\nu_x$, $r_0>0$ and $M>0$ such that, for all $0<r<r_0$ and all $x \in \partial U$,
\begin{equation}\label{domain2}
 d((x+\partial P_\nu) \cap \overline{B_r}, \partial U \cap \overline{B_r}) \leq Mr^2.
 \end{equation}
 
\smallskip

{\it The nonlinearity.} 
We assume that $F : \mathcal{M}^d \to \real$  
is uniformly elliptic, that is there exist $\Lambda>\lambda>0$ such that
\begin{equation}\label{op1}
 \lambda \text{tr}( N) \leq F(M)-F(M+N) \leq \Lambda \text{tr}( N) \ \hbox{ for all } M,N \in \mathcal{M}^d \ \hbox{ such that } N \geq 0,
\end{equation}
and homogeneous, that is 
\begin{equation}\label{op2}
F(0)=0.
\end{equation}

\smallskip

Typical examples of $F$'s satisfying \eqref{op1} and \eqref{op2} are  the convex Hamilton-Jacobi-Bellman  and the non convex Hamilton-Jacobi-Isaacs nonlinearities 
$$ F(M) = \sup_{\alpha \in \mathcal{A}}[ - \text{tr}(A^\alpha M)] \ \hbox{ with  } \ \lambda I \leq A^\alpha \leq \Lambda I \ \text{for all}  \ \alpha \in \mathcal{A} $$
and 
$$ F(M) = \inf_{\beta \in \mathcal{B}}\sup_{\alpha \in \mathcal{A}} [-\text{tr}(A^{\alpha\beta}M)] \ \hbox{ where } \ \lambda I \leq A^{\alpha\beta} \leq \Lambda I \ \text{ for all} \  (\alpha,\beta) \in \mathcal{A}\times \mathcal{B},$$
which arise respectively in  the theories  of stochastic control of diffusion processes  and  zero-sum stochastic differential games in both cases  without running cost.
 
\medskip

In fact all  $F$'s  satisfying  \eqref{op1} and \eqref{op2} have a max-min representation as above and, hence, are one-positively homogeneous, that is  
\begin{equation}\label{op3}
F(tM) = tF(M) \  \text{ for all} \  t>0 \ \text{ and } \ M \in \mathcal{M}^d.
\end{equation}

\medskip

{\it The boundary data for the Neumann problem.}  As far as the deterministic (non-random) boundary condition on $K$ is concerned we assume what is sufficient for  the Neumann problems \eqref{main00:N} and \eqref{mainN} to have unique solutions, that is
\begin{equation}\label{f}
f:K\to \real   \ \text{ is bounded and continuous.} 
\end{equation}

\medskip

{\it The random media and data.}  We make different assumptions depending on whether the boundary data is the restriction of a function defined on the whole space or is obtained by lifting from a hyperplane.  To avoid  repetition we state some properties for functions defined on $\real^n$ with $n=d$ in the former   and $n=d-1$ in the latter cases.

\medskip

We begin with the boundary data  $\psi : \real^n \times \Omega \to \real$ arising in the cell problems. We assume that 
\begin{equation}\label{takis-1}
\psi \ \text{  is  measurable with respect to} \   \mathcal{B}(\real^n) \otimes \mathcal{F},
\end{equation}
stationary in $(y,\omega)$ with respect to the group action $(\tau_y )_{y \in \real^d}$, that is, for all $y \in \real^n$,
  \begin{equation}\label{eqn: stationarity psi}
  \psi(y,\omega) = \psi(0,\tau_y\omega),
  \end{equation}
and bounded uniformly in $\omega \in \Omega$, that is there exists $C>0$ such that 
\begin{equation}\label{psi bdd}
|\psi(y,\omega)| \leq C \ \text{ for all} \  y \in \real^n \ \text{and } \omega \in \Omega. 
\end{equation}

\medskip

When $n=d$, which is  the case of the cell problem \eqref{eqn: cell},  we also assume that $\psi$ is Lipschitz continuous uniformly in $\omega$,  that is there exists  $C>0$ such that,   for all $y,w \in  \real^n$ and $\omega \in \Omega$, 
  \begin{equation}\label{eqn: continuity psi}
  |\psi(y,\omega)-\psi(w,\omega)| \leq C|y-w|;
   \end{equation}
we remark  that, with only minor changes, we may assume that $\psi$ is  H\"{o}lder or even just uniformly continuous in $y$. We leave it up to the interested reader to fill in the details.
\medskip

Furthermore, as mentioned in the introduction, we need some mixing assumption on the random field $\psi$.  We assume that 
\begin{equation}\label{hyp: mixing}
 \psi(\cdot,\cdot)  \  \text{ is \  $\phi$-mixing 
 with a rate $\phi$ such that} \    \rho = \int_0^\infty \phi(r)^{1/2} r^{d-2} dr <+\infty;
 \end{equation}  
 the value of this integral, $\rho$, will be considered a universal constant in the following sections.

 \medskip
 
 In the case that  the boundary data is assigned  by lifting from a hyperplane, that is  when $n = d-1$,  the condition on the boundary data is based on LSI.  
 
 \medskip
 
 We assume that there  exists  a collection of identically distributed $(X_j)_{j \in \integer^{d-1}}$ such that, if  $X=(X_j)_{j \in \integer^{d-1}}$, then 
 \begin{equation}\label{takis-2}
\text{ the measure $\text{law}(X)$ on $\real^{\integer^{d-1}}$ has LSI with constant $\rho$}
\end{equation}
 and
 \begin{equation} \label{psi3}
 \psi(y,\omega) =  X_j(\omega) \ \hbox{ when }  y \in j +[0,1)^{d-1};
 \end{equation}
this is just the random checkerboard boundary data. Note that in this case we only have $\integer^{d-1}$ stationarity. 

\medskip

Next we state the conditions on the random boundary data $g: \partial U \times \real^d \times \Omega \to \real$. We assume that
\begin{equation}\label{takis-3}
\text{ $g$ is measurable with respect to $\mathcal{B}(\real^d) \otimes \mathcal{B}(\real^d) \otimes \mathcal{F},$}
\end{equation}
and bounded and Lipschitz continuous uniformly in $\omega \in \Omega$, that is there exists  $C>0$ such that,   for all $(x,y),(z,w) \in \partial U \times \real^d$ and $\omega \in \Omega$, 
  \begin{equation}\label{eqn: continuity g}
   |g(x,y,\omega) | \leq C \ \hbox{ and } \ |g(x,y,\omega)-g(z,w,\omega)| \leq C(|x-z|+|y-w|) .
   \end{equation}
We remark again that, with only minor changes, we may assume that $g$ is  H\"{o}lder or even just uniformly continuous in $(x,y)$; we leave it up to the interested reader to fill in the details. 
\medskip

Finally we assume that, for each fixed $x_0 \in \partial U,$
\begin{equation}\label{stationary g}
g(x_0, \cdot, \cdot)  \ \text{ is stationary in $(y, \omega)$} 
\end{equation}
and 
\begin{equation}\label{takis-4}
\text{ $g(x_0, \cdot, \cdot)$  satisfies \eqref{hyp: mixing} with constants  uniform in $x_0.$}
\end{equation}

\subsection*{Some pde background}
We recall and prove  some background results from the theory of fully nonlinear uniformly elliptic equations that we will need for the proofs in the sequel. First we discuss \eqref{eqn: cell}  and  then \eqref{eqn: neumann cell0}. Then we introduce a selection criterion to identify uniquely a  particular solutions to 
 \eqref{eqn: cell}  and  then \eqref{eqn: neumann cell0} when $\psi$ is discontinuous.
 
 \smallskip

In what follows it is convenient to use  the extremal Pucci operators $\mathcal{P}^{\pm}_{\lambda,\Lambda}$ associated with $F$ defined by 
\begin{equation}\label{Pucci}
\mathcal{P}^+_{\lambda,\Lambda}(N):= \Lambda \text{tr}(N_+) - \lambda \text{tr}(N_-) \ \hbox{ and } \ \mathcal{P}^-_{\lambda,\Lambda}(N):= \lambda \text{tr}(N_+) - \Lambda \text{tr}(N_-);
\end{equation}
it follows from \eqref{op1} and \eqref{op2} that,  
for all $M,N\in \mathcal{M}^d$ ,
 \begin{equation}\label{eqn: ellipticity w/ pucci}
 -\mathcal{P}^+_{\lambda,\Lambda}(M-N)\leq F(M)-F(N) \leq -\mathcal{P}^-_{\lambda,\Lambda}(M-N).
 \end{equation}

\medskip

{\it The Dirichlet Problem   \eqref{eqn: cell}}.
The first two results are an oscillation decay lemma, which is a consequence of the interior Lipschitz estimates for solutions to elliptic pdes, and a boundary Lipshitz estimate;  for the proofs we refer to 
Caffarelli and Cabre \cite{CC95}.

\begin{lem}\label{lem:osc_decay}\textup{(Oscillation decay/Interior Lipschitz estimate)}
Assume \eqref{op1}, \eqref{op2}, \eqref{psi bdd} and \eqref{eqn: continuity g}   and let $v$ be the unique bounded solution of \eqref{eqn: cell}. Then, for $R>1$ and $y, z \in \partial P_{\nu}+R\nu,$ 
$$
\sup_{|y-z|\leq 1}|v(y, \omega) - v(z,\omega)| \leq  CR^{-1}{\rm osc}_{ \partial P_{\nu}} \psi. 
$$
\end{lem}

\begin{lem}\label{lem:boundary_holder}\textup{(Boundary Lipshitz)}
Assume \eqref{op1} and \eqref{op2}.  If $u \in C(\overline{B_1^+})$ solves   $F(D^2u) = 0$ in $B_1^+$ and $u=h \in C^{0,1}(\partial P_{e_d} \cap B_1),$ then $u \in C^{0,1}(\overline{B_{1/2}^+})$ and there exists $C=C(n,\lambda, \Lambda)>0$ such that 
$$
\|u\|_{C^{0,1}(\overline{B^+_{1/2}})} \leq C ( \sup_{B_1^+}|u| + \|h\|_{C^{0,1}(\partial P_{e_d} \cap B_1)}).
$$

\end{lem}

The next  lemma is an important tool for the proofs later in the paper, because it quantifies the error introduced by solving the Dirichlet problem in a half space with a cut-off applied to the boundary data.

\begin{lem}\label{lem: localization} \textup{(Localization)} Fix $\nu \in S^{d-1}$, let $M>0$ and $L>1$ and suppose that $v$ is a solution of
\begin{equation}\label{eqn: localization}
\left\{
\begin{array}{lll}
-\mathcal{P}^+_{\lambda,\Lambda}(D^2v) \leq 0  & \hbox{ in } & P_{\nu} \cap \textup{Cyl}_{\nu,L},\\ [1mm]
v \leq 0 & \hbox{ on } & \partial P_{\nu}\cap \textup{Cyl}_{\nu,L}, \\ [1mm]
v\leq M & \hbox{ on } & \partial \textup{Cyl}_{\nu, L} \cap \overline{P_{\nu}}.
\end{array}\right.
\end{equation}
Then 
$$ v \leq  2\tfrac{\Lambda}{\lambda}dML^{-1} \ \hbox{ in } \ \textup{Cyl}_{\nu,1}.$$
\end{lem}
\begin{proof}
Since the Pucci maximal operators are rotation invariant, without loss of generality, we assume that $\nu=e_d$ and, for notational simplicity, write $\textup{Cyl}_L$ for $\textup{Cyl}_{e_d, L}$.  
Consider the  barrier $\psi$
$$ \psi(x) := ML^{-2}(|x'|^2-2\tfrac{\Lambda}{\lambda}(d-1)x_d^2)+(2\tfrac{\Lambda}{\lambda}(d-1)+1)ML^{-1}x_d.$$
It is straightforward to check that $\psi \geq v$ on $\partial \textup{Cyl}_L$ and that $\psi$ is a smooth supersolution of $-\mathcal{P}^+_{\lambda,\Lambda}(D^2\psi) \geq 0 $ in $P_{e_d}$.  From the comparison  of viscosity solution  we have $v \leq \psi$ in $\textup{Cyl}_L$ and the claim follows, since 
$$ \psi \leq 2(\tfrac{\Lambda}{\lambda}(d-1)+1)ML^{-1} \ \hbox{ in } \ \textup{Cyl}_1.$$
\end{proof}

Next we recall a particular case of Theorem 1 from Armstrong, Sirakov and Smart \cite{ASS12}  which yields the existence of singular solutions to nonlinear uniformly elliptic equations.  
\begin{thm}\label{thm:fundamental}
Assume \eqref{op1} and \eqref{op2}. For any $\gamma \in (-1,1)$, there exists a unique constant\\  $\beta =\beta(d,\lambda,\Lambda,\gamma)>\max\{\tfrac{\lambda}{\Lambda}(d-1)-1,0\}$ such that the problem
\begin{equation}\label{eqn: barrier supersoln}
\left\{
\begin{array}{lll}
-\mathcal{P}^+_{\lambda,\Lambda}(D^2\Phi) = 0 & \hbox{ in } & \textup{Cone}_\gamma :=  \{x\in \real^d : x_d > \gamma|x|\}, \\[1mm]
\Phi = 0 & \hbox{ on } & \partial \textup{Cone}_\gamma\setminus\{0\},
\end{array}\right.
\end{equation}
has a unique positive solution $\Phi \in C(\overline{\textup{Cone}}_\gamma\setminus\{0\})$ which is $-\beta$ homogeneous, that is
$ \Phi(x) = |x|^{-\beta}\Phi(x/|x|),$
and satisfies  $\max_{\partial B_1 \cap K_\gamma} \Phi = 1$. 
\end{thm}

When $\gamma = 0$, that is in the half space case,  Leoni \cite{leoni12} shows that
\begin{equation}\label{eqn: beta}
 \frac{\lambda}{\Lambda}(d-1) \geq \beta \geq \frac{\lambda}{\Lambda}d-1;
 \end{equation}
from now on we will always use $\beta$ to refer to the half-space singular solution exponent for the ellipticity ratio $\Lambda/\lambda$.

\medskip

It will be useful  for later to estimate, in the half-space case, the decay of $\Phi$  in the directions $x'$ which are orthogonal to $e_d$.  
\medskip

We first note that  Lemma \ref{lem:boundary_holder} and a simple covering argument yield  that $\Phi$ is Lipschitz with universal constant in $(B_2 \setminus B_{1/2}) \cap K_0$.
 
 \medskip
  
Next  we show that $\Phi({x}/{|x|}) \sim x_d/{|x|}$ and, therefore,
\begin{equation}\label{eqn: decay in tang}
\Phi(x) = |x|^{-\beta}\Phi(x/|x|) \sim x_d/{|x|^{\beta+1}}.
\end{equation}
The Lipschitz estimate and the fact that $\Phi = 0$ on $\partial K_0$ yield 
$$ \Phi(\frac{x}{|x|}) \leq C \hbox{dist}\left(\frac{x}{|x|}, B_2 \setminus B_{1/2} \cap \partial K_0\right) \leq C \frac{x_d}{|x|}.$$
For the other direction we note that $A= \min_{B_2  \cap \{x_d >1/2\}} \Phi>0$ is universal due to the normalization $\max_{\partial B_1 \cap K_0}\Phi = 1$ and the Harnack inequality. Then, for each $x$, let $x' = x-x_de_d$ and $y = \frac{3}{4}e_d+\frac{x'}{|x|}$,  consider the radially symmetric barrier $v$ satisfying
\begin{equation*}
\left\{
\begin{array}{lll}
-\mathcal{P}^+_{\lambda,\Lambda}(D^2v) = 0 & \hbox{ in } & B_{1/2}(y) \setminus B_{1/4}(y), \\
v = 0 & \hbox{ on } & \partial B_{1/2}(y), \\
v = A & \hbox{ on } & \partial B_{1/4}(y), \\
\end{array}\right.
\end{equation*}
and observe that the  maximum principle yields $\Phi \geq v$ in $B_{1/2}(y) \setminus B_{1/4}(y)$. Meanwhile, an explicit computation shows that 
$$ v(\frac{x}{|x|}) \geq cA \frac{x_d}{|x|},$$
and, hence,
$$ \Phi(\frac{x}{|x|}) \geq v(\frac{x}{|x|}) \gtrsim \frac{x_d}{|x|},$$
which proves both directions of \eqref{eqn: decay in tang}.

\medskip

{\it The Neumann Problem \eqref{eqn: neumann cell0}}.  
The first  result is a localization result analogous to Lemma \ref{lem: localization}.

\begin{lem} \label{localization:N}  \textup{(Localization)} Assume \eqref{op1} and \eqref{op2}, fix  $\nu\in S^{d-1}$ and, for $L>0$, \\  let  $u: \Pi_\nu \to \R$ satisfy
$$
\left\{\begin{array}{ll}
-\mathcal{P}^+_{\lambda,\Lambda}(D^2 u) \leq 0,&\hbox{ in } \    \Pi_\nu\\[1mm]
\partial_{\nu} u  \leq 0 &\hbox{ on } \   \partial  P_\nu \cap   B_L ,\\[1mm]
u\leq 0 &\hbox{ on } \    (\partial P_\nu+\nu) \cap B_L\\[1mm]  
u \leq M &\hbox{ on } \  \Pi_\nu \cap \partial B_L.
\end{array}\right.
$$
 Then
$$
|u| \leq (1+(d-1)\frac{\Lambda}{\lambda})\frac{M}{L^2} \ \ \text{on}  \ \ \Pi_\nu \cap \{ |x'| \leq 1\} 
$$
\end{lem}

\begin{proof}
The claim follows from the comparison principle using the barrier 
$$
\varphi(x) = \frac{M}{L^2}\left(|x'|^2+(d-1)\frac{\Lambda}{\lambda} (1-(x\cdot\nu)^2)\right).
$$

\end{proof}
The next result is about up to the boundary regularity of solutions of the Neumann problem. Its proof can be found in Milakis and Silvestre  \cite{MS06}.
\begin{lem}\label{lem:boundary_holder:N}(Boundary $C^{1,\alpha}$-regularity)
Assume \eqref{op1} and \eqref{op2}. For every $\alpha \in (0,1),$ there exists $C=C(d,\lambda,\Lambda,\alpha)>0$ such that, if $v$ solves the Neumann problem
$$
\left\{\begin{array}{ll}
F(D^2v) = 0 &\hbox{ in } B_1^+,  \\[1mm]
 \partial_{x_d} v = h &\hbox{ on }  \partial P_{e_d} \cap B_1,
 \end{array}\right.
$$
then
$$
\|v\|_{C^{0,\alpha}(B_{1/2}^+)} \leq \sup_{B_1}|v| + C\max_{\partial P_{e_d} \cap B_1} |h|. 
$$
 Furthermore,  there is some $\alpha=\alpha(d,\lambda, \Lambda) \in (0,1)$ such that, 
$$
\|v\|_{C^{1,\alpha}(B^+_{1/2})} \leq \sup_{B_1} |v| + C\|h\|_{C^{0,\alpha}({\partial P_{e_d} \cap B_1})}.
$$
\end{lem}

\medskip

As for the decay of oscillations, the estimate is slightly better than for the  Dirichlet problem,  since the oscillations decay up to the boundary.

\begin{lem}\label{lem:osc_decay:N}\textup{(Oscillation decay)}
Assume  \eqref{op1} and \eqref{op2} and  fix $\nu \in S^{d-1}$.  If $v_R$ is  the solution to \eqref{eqn: neumann cell0}, then,
$$
\sup_{|y-z|\leq 1}R^{-1}|v_R(y, \omega) - v_R(z,\omega)| \leq  CR^{-1}\sup_{ \partial P_{\nu}} | \psi | \ \ \text{if } \ \  y  \in   \overline \Pi_\nu^R+\frac{1}{2}R\nu  
$$
and there exists $0<\alpha = \alpha(d, \lambda, \Lambda)$ such that
$$
\sup_{|y-z|\leq 1}|\partial_{\nu}(v_R(y, \omega) - v_R(z,\omega))| \leq  CR^{-\alpha}{\rm osc}_{\partial P_\nu}\psi(y)\quad \hbox{ if  }  \  y  \in    \overline \Pi_\nu^{R}+\frac{1}{2}R\nu.
$$
\end{lem}
\begin{proof}
To prove the first inequality we consider  the rescaled function
$\tilde{v}(y,\omega):=\frac{1}{R} v_R(Ry,\omega) $ and apply Lemma~\ref{lem:boundary_holder:N} to  the Neumann problem in $B_2 \cap P_\nu$ for $|y-z| \leq R^{-1}$.
\smallskip

For the second inequality we apply the interior $C^{1,\alpha}$-regularity result for  solutions to uniformly elliptic operators (see \cite{CC95}) to $\tilde{v}$ and use the fact that 
$
\osc_{2\Pi_\nu} \tilde{v} \leq 2 \osc_{\partial P_\nu} \psi.
$
\end{proof}

\medskip

Finally we introduce the analogue of the singular solutions Theorem~\ref{thm:fundamental} in the Neumann case. Let 
\begin{equation}\label{fundamental:N}
\Phi(x) = |x|^{-\beta}\hbox{ with } \beta = \frac{\lambda}{\Lambda} (d-1)-1.
\end{equation} 

\smallskip
 
An explicit calculation shows  that  $\phi(x) = \Phi(x+e_d)$ satisfies
\begin{equation}\label{singular:N}
\left\{
\begin{array}{lll}
-\mathcal{P}^+_{\lambda,\Lambda}(D^2\phi) = 0 & \hbox{ in } & P_{e_d}, \\[1mm]
-\partial_{e_d}\phi \geq  c_d(1+|x|)^{\frac{\lambda}{\Lambda} (d-1)} & \hbox{ on } & \partial P_{e_d}.
\end{array}\right.
\end{equation} 

\medskip

 {\it Discontinuous Boundary Data.  } We discuss the Dirichlet problem here, but similar arguments apply to the Neumann problem.  In the course of the proof of the homogenization of the cell problem, we will have to consider \eqref{eqn: cell} with bounded but discontinuous boundary data. Such boundary value problems  may not have, in general,  a unique solution unless $F$ has some additional structure. It was shown, for example, in  \cite{Feldman13} that uniqueness holds if the discontinuities are concentrated on a subset of $\partial D$ of Hausdorff dimension less than $d-1$ and if $|\lambda/\Lambda -1|$ is small depending on the dimension of the discontinuity set.   
 
\medskip

 We describe next a selection mechanism we will be using throughout the paper, which allows to talk about a unique solution to \eqref{eqn: cell} satisfying the comparison principle

\medskip

 Since the argument is more general than \eqref{eqn: cell}, we consider the boundary value problem
\begin{equation}\label{takis5}
\begin{cases}
F(D^2u)=0 \ \text{in } \ D,\\[1mm]
u=\phi \ \text{on} \ \partial D;
\end{cases}
\end{equation}
here $D$ is a general domain, which typically in this paper will be $P_{\nu}$ for some $\nu \in S^{d-1}$, and $\phi:  \partial D \to \real$ is bounded but possibly discontinuous. In what follows  we  write $u(\cdot; \phi)$ to denote the solution of the boundary value problem with data $\phi$.
\medskip

A maximal/minimal solution to \eqref{takis5} can be constructed by the classical Perron's method (see, for example, \cite{CIL}). Here we describe an alternative method to select in a unique fashion a solution $u(\cdot; \phi)$, which  is defined as the pointwise  in {$\overline D$} and local uniform in $D$ limit of the (unique) solution $u(\cdot ;\phi_\delta)$ of \eqref{takis5},  for some  appropriately  defined regularization $\phi_\delta$ of $\phi$. Moreover, $u(\cdot; \phi)$ satisfies  the contraction property, that is, 
for any two bounded possibly discontinuous $\phi,\phi'$,
\begin{equation}\label{takis6}
\sup_{D} (u(\cdot;\phi) - u(\cdot;\phi'))_\pm  \leq \sup_{\partial D} (\phi-\phi')_\pm;
\end{equation}
in what follows we say that the so defined  $u(\cdot;\phi)$ satisfies the maximum principle.

\medskip

To this end, given a bounded  possibly discontinuous $\phi:\partial D\to \real$, for $\delta>0$,  we define
$$ \phi_\delta(y): = \max_{z \in \partial D} [ \phi(z) - \frac{1}{\delta}|y-z|].$$
It is immediate that  $ \phi_\delta$  is Lipschitz continuous, $\phi_\delta \geq \phi$ and  the   $ \phi_\delta$'s  decrease, as $\delta \to 0$,  to $\phi$ at every point of continuity of $\phi$.  The standard comparison principle for bounded uniformly continuous solutions to \eqref{takis5} yields that the solutions $u(\cdot ;\phi_\delta)$'s  are  decreasing in $\delta$ and, hence,   for each $y\in \overline D$, we may define  $u(y;\phi)$ by $u(y;\phi):=\lim_{\delta \to 0}u(y;\phi_\delta)$ pointwise in $\overline D$ and locally uniformly  in $D$.  Since the  regularization of $\phi$ is contractive, that is, for any $\phi,\phi'$ as above, $\sup_{\partial D}(\phi_\delta - \phi'_\delta)_\pm \leq \sup_{\partial D}(\phi -\phi')_\pm,$
it follows, again from the comparison principle, that 
$$\sup_{D}(u(\cdot;\phi_\delta) - u(\cdot;\phi'_\delta))_\pm \leq \sup_{\partial D}(\phi_\delta - \phi'_\delta)_\pm \leq \sup_{\partial D} (\phi-\phi')_\pm;$$
letting $\delta \to 0$ on the left hand side above yields \eqref{takis6}. 

\medskip

\subsection*{The results} Now that we have set up the assumptions, we give the precise statements of the main results, Theorems \ref{thm:main0}, \ref{thm:main01} and \ref{thm:main02} from the Introduction.  
\medskip

{ \it The Dirichlet cell problem \eqref{eqn: cell}. }  
The first result assumes convexity/concavity and imposes no restrictions on the ellipticity ratio of $F$, while the second applies to general $F$'s but requires  a dimension dependent upper bound on $\Lambda/\lambda$.

\medskip

In what follows $\beta(\lambda,\Lambda)$ is the homogeneity exponent of the singular solution for the ellipticity class of $F$ in the half plane (see \eqref{eqn: beta}).

\smallskip

\begin{thm}\label{thm: convex}
Suppose that, in addition to \eqref{op1}and  \eqref{op2},  $F$ is either convex or concave and $\psi$ satisfies  \eqref{takis-1}, \eqref{eqn: stationarity psi}, \eqref{psi bdd},  \eqref{eqn: continuity psi} and   \eqref{hyp: mixing}  .  Then, for each $\nu \in S^{d-1}$,  the cell problem \eqref{eqn: cell} has a unique solution $v_\nu$ which concentrates about its mean   with rate
\begin{equation}
\P(\{ \omega : |v_\nu(y'+R\nu,\omega)-\E v_\nu(y'+R\nu)| >t\}) \leq C \exp (-cR^{\hat \beta}t^2) \ \hbox{ for } \ y' \in \partial P_\nu  \ \text{and} \ t>0.
\end{equation}
where $\hat \beta = \beta(\lambda,\Lambda).$  Moreover, the limit  $\mu(\psi,F,\nu):=\lim_{R\to\infty} \E v_\nu(R\nu)$ exists and equals the almost sure  limit of $v_\nu(R\nu,\omega)$, as $R\to\infty$, and, furthermore,  for some universal constant $C = C(d,\lambda,\Lambda,\rho) = C(d,\lambda,\Lambda)\rho^2$ and all $y' \in \partial P_\nu$, 
\begin{equation}\label{eqn: exp conv disc}
 |\E v_\nu(y'+R\nu) - \mu(\psi,F,\nu) | \leq C(\log R)^{1/2}R^{-\hat \beta /2}.
  \end{equation}
\end{thm}

\medskip

\begin{thm}\label{thm: nonconvex}
Assume  that $F$ and $\psi$ satisfy  \eqref{op1}, \eqref{op2}, \eqref{takis-1}, \eqref{eqn: stationarity psi}, \eqref{psi bdd},  \eqref{eqn: continuity psi} and   \eqref{hyp: mixing} and assume that $ 2  \beta (\lambda, \Lambda) -(d-1)>0$.  Then the results of Theorem \ref{thm: convex} hold with  $\hat \beta = 2\beta (\lambda, \Lambda) - (d-1)$.
\end{thm}

\medskip

Under the assumptions that lead to the homogenization of the cell problem,  the ergodic constant (homogenized boundary condition) is continuous  with respect to the normal direction and the boundary data. This a very important property to establish homogenization in general domains.  Its proof relies in a critical way on having a rate of convergence for the cell problem.

\begin{thm}\label{continuity00}
Assume \eqref{op1},  \eqref{op2} and $ 2\beta(\lambda,\Lambda)-(d-1)>0$ if $F$ is neither convex nor concave. 
\begin{enumerate}[(i)]
\item If  $\psi$ and $\psi'$ satisfy \eqref{takis-1}, \eqref{eqn: stationarity psi}, \eqref{psi bdd},  \eqref{eqn: continuity psi} and   \eqref{hyp: mixing} and fix $\nu \in S^{d-1}$, then 
$$|\mu(\psi,F,\nu)-\mu(\psi',F,\nu)| \leq \|\psi-\psi'\|_{\infty, \real^d \times \Omega}. $$
\item There exists $C=C(d,\lambda,\Lambda) >0$ such that, for every $\nu, \nu'  \in S^{d-1}$ and   $\alpha \in (0, \alpha')$ with\\
 $\alpha' = \alpha'(\hat{\beta}) := \frac{\hat{\beta}}{2(1+\hat{\beta})}$, if $\psi$ satisfies  \eqref{takis-1}, \eqref{eqn: stationarity psi}, \eqref{psi bdd},  \eqref{eqn: continuity psi} and   \eqref{hyp: mixing}, then 
$$ |\mu(\psi,F,\nu)-\mu(\psi,F,\nu')| \leq C(\osc \psi)(1+\|D \psi\|_\infty)^{\alpha'}|\nu-\nu'|^{\alpha}.$$
\end{enumerate}
\end{thm}

\medskip

Lastly we have the following homogenization result for the Dirichlet problem in general domains with an algebraic rate in probability.  
\begin{thm}\label{thm: main}
Assume \eqref{domain1}, \eqref{op1} and \eqref{op2} and define $\hat\beta$ as in either Theorem~\ref{thm: convex} or Theorem~\ref{thm: nonconvex} depending on whether $F$ is convex/concave or not. In addition suppose that the boundary data $g$ satisfies \eqref{takis-3}, \eqref{eqn: continuity g}, \eqref{stationary g}  and \eqref{takis-4}. Let $u^\e$ and  $\overline{u}$ be respectively the solutions to \eqref{main0} and \eqref{main} with Dirichlet boundary data $g(x,x/\ep,\omega)$ and $\overline g(x) := \mu(g(x,\cdot, \cdot),F,\nu_x)$ respectively. 
The $u^\e$ concentrates about its mean in the sense that, for   $ \alpha_0:=\frac{\hat \beta }{4+3\hat \beta }$ and every $p>0$, there exists a sufficiently large universal constant $M_p$ such that
\begin{equation}
 \P(\{ \omega : \sup_{\{x: d(x,\partial U)>\e^{1-2\alpha_0/\hat \beta }\}} |u^\e(x,\omega)-\E u^\e(x)|> M_p(\log\tfrac{1}{\e})^{1/2}\e^{\alpha_0}\}) \lesssim \e^p,
  \end{equation}
and the expected value  $\E u^\e$ converges, as $\e \to 0$,  to $\overline{u}$ with the rate 
\begin{equation}\label{eqn: a est}
\sup_{\{x: d(x,\partial U)>\e^{1-2\alpha_0/\hat \beta} \}}|\E u^\e(x) - \overline{u}(x)| \lesssim (\log\tfrac{1}{\e})^{1/2} \e^{\alpha_0}.
\end{equation}

\end{thm}

\medskip

{\it The lifting from the hyperplane plane.  }We describe the general setting for the Dirichlet cell problem arising when assigning boundary data by ``lifting from the hyperplane" as in \eqref{eqn: prj up} and then state the result which asserts that the cell problem homogenizes for general $F$ without either a convexity/concavity assumption or a restriction on the ellipticity ratio.
 
\medskip

We assume that 
\begin{equation}\label{diff0}
\text{ $\zeta_0$ is a diffeomorphism  from an open subset of $\real^{d-1}$ to a relatively open subset of $\partial U$,}
\end{equation}
and we extend $\zeta_0$ to a local diffeomorphism $\zeta : \real^d \to \real^d$ by
\begin{equation}\label{diff1}
\zeta(x) = (\zeta_0(x'),x_d\nu_{\zeta(x')}).
\end{equation}

Given a random field $\psi: \real^{d-1}\times \Omega \to \real$, we define the boundary data on the general domain  $U$ by
$$ g^\e(x,\omega) := \psi(\e^{-1}\zeta^{-1}(x),\omega) \ \hbox{ for } \ x \in \partial U;$$
note that here we implicitly extend $\psi$ to map $\psi: \partial U \times \real^d \times \Omega \to \real$ which is  constant in the $y_d$ direction.
\medskip

We study the asymptotic behavior of the rescaling $v^\e(y):= u^\e(x_0 + \e y)$ of $u^\e$, the solution to \eqref{main0} with $g$ as above,  at some $x_0$ in the image of $\zeta_0$, which solves  the boundary value problem which solves the boundary value problem

\begin{equation}
\left\{
\begin{array}{lll}
F(D^2v^\e) = 0 & \hbox{ in } &  \e^{-1}(U-x_0) \\[1mm]
v^\e (y,\omega) = \psi(\e^{-1}\zeta^{-1}(x_0)+(D\zeta^{-1})(x_0)y + w_\e(y),\omega)  & \hbox{ on } &  \e^{-1}\partial(U-x_0);
\end{array}\right.
\end{equation}
with $|w_\e(y)|\leq C\e|y|^2$ for some $C=C(\|\zeta^{-1}\|_{C^2})$.

\medskip
 
The corresponding cell problem is then to solve
\begin{equation}\label{eqn: other cell}
\left\{
\begin{array}{lll}
F(D^2v_{T}) = 0 & \hbox{ in } &  T P_{e_d}, \\[1mm]
v_{T} (y,\omega) = \psi(T^{-1}y,\omega)  & \hbox{ on } &   T \partial  P_{e_d}, 
\end{array}\right.
\end{equation}
where $T : \real^{d} \to \real^{d}$ is an invertible linear map, and to show that there exists an ergodic  constant  $\tilde{\mu}(\psi,F,T)$  such that,  almost surely,  
$$ \tilde{\mu}(\psi,F,T) = \lim_{R \to \infty } v_{T}(R\nu,\omega).$$
We can further normalize $T$ so that $Te_d$ is a unit vector orthogonal to $T\partial P_{e_d}$ as this will not change the definition of $v_T$.  We remark that, at the expense of changing $F$,  we may reduce to the case of $T = I$ by changing variables to $z = T^{-1}y$. Notice that  the ellipticity constants will remain bounded as long as we work with a class of maps with $T$ and $T^{-1}$ bounded. For example, if $T = D\zeta (\zeta^{-1}(x_0))$, the bound will depend only on the properties of the diffeomorphism $\zeta$.  We emphasize that such a transformation need not change $\beta$, since, in this case,  we will use $\Phi(T\cdot)$, which has the same homogeneity as $\Phi$, as a supersolution barrier for the transformed problem.

\medskip

The result  for \eqref{eqn: other cell} is the following theorem.
\medskip

\begin{thm}\label{thm: lsi cell}
Assume that $F$ and $\psi$ satisfy  \eqref{op1}, \eqref{op2},  \eqref{takis-1}, \eqref{eqn: stationarity psi}, \eqref{psi bdd}, and the log-Sobolev assumptions \eqref{takis-2}  and \eqref{psi3}.  Then the conclusions of Theorem~\ref{thm: convex} holds for $v_{T}$ with $\hat \beta = \min\{\beta(\lambda,\Lambda),2\}$ and  constants that depend on $\lambda, \Lambda$ and the bounds for $T$ and $T^{-1}$ .
\end{thm}

The proof  of Theorem~\ref{thm: lsi cell} is based  on the concentration inequalities for measures with log-Sobolev inequality and does not depend on the concentration inequalities for the strong-mixing setting Theorems~\ref{concentration1} and \ref{concentration2}. As a result no restrictions are imposed on the ellipticity ratio of $F$.  We expect the a proof very similar to the one of Theorem~\ref{thm: main} will also give a homogenization for general domains with data locally constructed by lifting from hyperplanes.  Note that in order to prove the continuity of $\tilde{\mu}$ with respect to $T$, which is required for the general domain result, we actually need to take a regularized version of the random checkerboard boundary data \eqref{psi3} so that \eqref{eqn: continuity psi} holds.  We leave the details to the reader.

\medskip

{\it The Neumann cell problem \eqref{eqn: neumann cell0}.}  The result when  $F$ either convex or concave is:
\begin{thm}\label{thm: convex N}
Assume that, in addition to \eqref{op1} and \eqref{op2},  $F$ is either convex or concave and $\psi$ satisfies   \eqref{takis-1}, \eqref{eqn: stationarity psi}, \eqref{psi bdd},  \eqref{eqn: continuity psi} and   \eqref{hyp: mixing} .  Then, for each $\nu \in S^{d-1}$,  the cell problem \eqref{eqn: neumann cell0} has a unique solution $v_{\nu,R}$ which concentrates about its mean with rate 
\begin{equation}
\P(\{ \omega :R^{-1}|v_{\nu,R}(y'+R\nu,\omega)-\E v_{\nu,R}(y'+R\nu)| >t\}) \leq C \exp (-cR^{\hat \beta}t^2) \ \hbox{ on  } \partial P_\nu  \ \text{and} t>0. 
\end{equation}
with $\hat \beta = \frac{\lambda}{\Lambda}(d-1)$.  Moreover, the limit  $\mu(\psi,F,\nu):=\lim_{R\to\infty} R^{-1}\E v_{\nu,R}(R\nu)$ exists and equals the almost sure limit of $R^{-1}v_{\nu,R}(R\nu,\omega)$, as $R\to\infty$. Furthermore, for some universal constant $C = C(d,\lambda,\Lambda,\rho) = C(d,\lambda,\Lambda)\rho^2$ and all $y' \in \partial P_\nu,$ 
\begin{equation}\label{eqn: exp conv disc N}
 |\E v_\nu(y'+R\nu) - \mu(\psi,F,\nu) | \leq C(\log R)^{1/2}R^{-\hat \beta /2}. 
  \end{equation}
\end{thm}

In the nonconvex/nonconcave  case the result is:
\begin{thm}\label{thm: nonconvex N}
Assume that $F$ and $\psi$ satisfy \eqref{op1}, \eqref{op2},  \eqref{takis-1}, \eqref{eqn: stationarity psi}, \eqref{psi bdd},  \eqref{eqn: continuity psi} and   \eqref{hyp: mixing} and $ \frac{\lambda}{\Lambda} >\frac{1}{2}$.  Then the conclusions  of Theorem \ref{thm: convex N} hold with  $\hat \beta = (\frac{\lambda}{\Lambda}-\frac{1}{2})(d-1)$.
\end{thm}

As for the Dirichlet problem, under the assumptions that lead to the homogenization of the cell problem,  the ergodic constant is continuous  with respect to the normal direction and the boundary data. This an very important property to establish homogenization in general domains.  Its proof relies in a critical way on the rate of convergence for the cell problem.

\begin{thm} 
\label{continuity:N}
Assume \eqref{op1},  \eqref{op2}  and  $\frac{\lambda}{\Lambda} >\frac{1}{2}$ if $F$ is neither convex nor concave.
\begin{enumerate}[(i)]
\item If  $\psi$ and $\psi'$ satisfy  \eqref{takis-1}, \eqref{eqn: stationarity psi}, \eqref{psi bdd},  \eqref{eqn: continuity psi} and   \eqref{hyp: mixing}, then,  for every $\nu \in S^{d-1}$, 
$$|\mu(\psi,F,\nu)-\mu(\psi',F,\nu)| \leq \|\psi-\psi'\|_{\infty, \real^d \times \Omega}. $$
\item There exists $C=C(d,\lambda,\Lambda) >0$ such that, for every $\nu, \nu'  \in S^{d-1}$ and  $\alpha \in(0,\alpha')$  with  $ \alpha'(\hat{\beta}) := \frac{\hat{\beta}}{2(1+\hat{\beta})},$ if  $\psi$ satisfies  \eqref{takis-1}, \eqref{eqn: stationarity psi}, \eqref{psi bdd},  \eqref{eqn: continuity psi} and   \eqref{hyp: mixing}, then  
$$ |\mu(\psi,F,\nu_1)-\mu(\psi,F,\nu_2)| \leq C |\psi|_{C^{0,\alpha}})\|\psi\|_\infty|\nu_1-\nu_2|^{\alpha}.$$
\end{enumerate}
\end{thm}

Finally we have the following  homogenization result for  the general domain problem whenever the cell problem homogenizes.

\begin{thm}\label{thm: main N}
Assume \eqref{K} and \eqref{domain1} and suppose that the assumptions of either Theorem~\ref{thm: convex N} or Theorem~\ref{thm: nonconvex N} hold and define $\hat\beta$ accordingly.  Let $u^\e$ and  $\overline{u}$ be respectively the solutions to \eqref{main00:N} and \eqref{mainN}, with Neumann data $g(x,x/\ep,\omega)$ and $\overline g(x) := \mu(g(x,\cdot, \cdot),F,\nu_x)$.  The $u^\e$  concentrates about its mean in the sense that,  there exists $\alpha_0(\hat{\beta}) >0$ and every $p>0$,  there exists a sufficiently large universal constant $M_p$ such that
\begin{equation}
 \P(\{ \omega :\sup_{x \in { U \setminus K}} |u^\e(x,\omega)-\E u^\e(x)|> M_p(\log\tfrac{1}{\e})^{1/2}\e^{\alpha_0}\}) \lesssim \e^p,
 \end{equation}
and the expected value $\E u^\e$ converges, as $\e \to 0$,  to $\overline{u}$ with rate 
\begin{equation}\label{eqn: a est N}
\sup_{x \in {U \setminus K}}|\E u^\e(x) - \overline{u}(x)| \lesssim (\log\tfrac{1}{\e})^{1/2} \e^{\alpha_0}.
\end{equation}
\end{thm}
\medskip

We remark that the analogue of Theorem~\ref{thm: lsi cell} will hold in the Neumann case as well although we do not provide the proof as it is a natural adaptation of the other proofs presented.  Again the extension to a result in general domains with Neumann data given by ``lifting up from the hyperplane" should also follow with similar arguments.

\section{The Dirichlet Cell Problem}\label{cell}
Here we present the proofs of Theorem~\ref{thm: convex}, Theorem~\ref{thm: nonconvex},  Theorem~\ref{continuity00} and Theorem~\ref{thm: lsi cell} which are about the existence and properties of the homogenized boundary condition or ergodic constant, that is the asymptotic behavior, as $R \to \infty$, of the solutions to Dirichlet cell problem  \eqref{eqn: cell} and the continuity with respect to the normal directions. 
\medskip

In many places throughout the section we will consider \eqref{eqn: cell}  with discontinuous data. In this case, when we talk about the solution satisfying the maximum/comparison principle  we refer to the one constructed  in the previous section.

\medskip
Since the arguments are rather long, we have divided the section  into three major parts.  The first is about  Theorem~\ref{thm: convex} and Theorem~\ref{thm: nonconvex}, the second deals with Theorem~\ref{continuity00} and the third concerns  Theorem~\ref{thm: lsi cell}.

\medskip

\subsection*{The proofs of Theorem~\ref{thm: convex} and Theorem~\ref{thm: nonconvex}.}
There are two main steps here. The first is  to consider, for $R$ large,  $v_\nu(R\nu,\omega)$ as a function of the boundary data and prove a Lipschitz estimate in order to apply one of the concentration inequalities given in Theorem~\ref{concentration1} and Theorem~\ref{concentration2}. The argument is essentially deterministic and will be the focus of one of the subsections below. In the second step we obtain  a quantitative estimate for the concentration of $v_\nu(y,\omega)$ and  show the convergence of the means $\E v_\nu(y)$, as $y\cdot \nu \to \infty$, to a deterministic constant $\mu(\psi,F,\nu)$. To motivate the arguments we begin with an outline of the proof. 
  
 \medskip

{\it The outline of the proof.} \ 
We consider the cell problem \eqref{eqn: cell}.  For simplicity we take $\nu = e_d$ and boundary data $\psi(y,\omega) = \xi_{y\bmod \integer^{d-1}}(\omega)$, where  $(\xi_k)_{k \in \integer^{d-1}}$ is a stationary ergodic field on $\integer^{d-1}$ with $\mu = \E \xi_0$ and $\sigma^2 = \textup{var}(\xi_0)$. 
Let $v(\cdot; \psi)$ be the solution of the cell problem \eqref{eqn: cell} (note that for notational simplicity we write $v$ instead of  $v_{e_d}$). Then, for large $R$,  we consider the value  $v_R(\psi) = v(Re_d ; \psi)$  as a function $v_R : [-1,1]^{\integer^{d-1}} \to \real$. 

\medskip

To keep the ideas simple, we first describe the argument when the interior equation  is just the Laplacian, $-\Delta$.  Then,  as we already said in the introduction, we write $v_R(\psi)$ as 
$$ v_R(\psi) = \int_{\partial P_{e_d}} P(Re_d,y) \psi(y) dy = \sum_{k \in \integer^{d-1}} a_k  \xi_k,$$
where $P:  P_{e_d} \times \partial P_{e_d}  \to \real$ is the Poisson kernel for the upper half space and 
$$ a_k = a_k(R) := \int_{k+[0,1)^{d-1}} P(Re_d,y)  \ dy,$$
and observe that 
$$\E v_R(\psi) = (\E\xi_0 ) \int_{ \partial P_{e_d}} P(Re_d,y)  \ dy = \mu,$$
and 
$$ a_k   \sim R(R^2+|k|^2)^{-\frac{d}{2}}.$$

\medskip

In particular we see that $v_R(\psi)$ has the form of an ergodic average and, hence, we can apply the ergodic theorem to get
$$ v_R(\psi) - \mu \to 0 \ \  \hbox{ almost surely  as } R \to \infty.$$
When the $\xi_k$'s  are i.i.d., a standard variance estimate yields
$$ \textup{var} ( v_R(\psi)) = \sigma^2\sum_{k \in \integer^{d-1}} a_k^2 \leq \int_{ \partial P_{e_d}} P(Re_d,y)^2 \ dy \leq C_d R^{1-d},$$
which already implies, by Chebyshev's inequality, homogenization in probability.  

\medskip

This is not, however, optimal, since it is possible to  obtain, using Hoeffding's inequality  (see \cite{hoeffding}), the following Gaussian-type concentration about the mean: 
\begin{equation}\label{eqn: hoeffding}
 \P(\{ \omega :|v_R(\psi)-\mu| > t \}) \leq 2 \exp\left(\frac{-t^2}{2\sum_{k \in \integer^d} a_k^2} \right) \leq 2 \exp\left( - c_dR^{d-1}t^2\right).
 \end{equation}
\medskip

Next we describe a way  to generalize the main components of the above argument to the nonlinear setting. 
  
\medskip

Arguing for the moment heuristically, we  suppose that we can linearize $v_R$ around $\psi \equiv 0$ and write
$$ v_R(\psi) = \sum_{k \in \integer^{d-1}} a_k \xi_k + o(|\xi|_{\ell^\infty}),$$
where $a_k: = \partial_{\xi_k}v_R(0)$ and $|\xi|_{\ell^\infty}$ is small.  
\medskip

Fix $j \in \integer^{d-1}$. Taking as boundary data $\xi_k = \indicator_{\{j\}}$  and  using the  strong maximum principle, the positive homogeneity of $F$ and Taylor's expansion, we find 
$$ 0 < v_R(\psi) = \eta^{-1}v_R(\eta \psi) = a_j + \eta^{-1} o(\eta),$$
and, after sending $\eta \to 0$,  $a_j > 0$.  
\medskip

Taking next $\psi \equiv 1$, in which case $v_R(\psi) = 1$, the above argument  yields, after letting $\eta \to 0$, 
$$ |a|_{\ell^{1}} = \sum_{k \in \integer^{d-1}} a_k = 1.$$
To show homogenization then, it is enough to obtain a  vanishing, as $R \to \infty$, bound on the  $\ell^2$-norm of $a_k$,  and,  since, 
$$ \sum_{ k \in \integer^{d-1}} a_k^2 \leq |a|_{\ell^{\infty}}|a|_{\ell^1} \leq |a|_{\ell^{\infty}}, $$
it suffices  to show that, as $R \to \infty$,    $|a|_{\ell^{\infty}} \to 0$.  
\medskip

Of course this depends on a concentration estimate of the form \eqref{eqn: hoeffding} holding in the nonlinear case.  This is a delicate issue related to the difference between the various concentration results stated in Section~\ref{prelim}.

\medskip

This is where the structure of the pde is needed to complete the argument. It is, however, evident from the heuristics above, that,  to obtain the result, it is enough to get bounds on $v_R(\psi)$ for boundary data of the form $\indicator_{k+[0,1)^{d-1}}$.

\medskip

For nonlinear equations we are not aware of any quantity other than $v_R(\psi)$ which controls the homogenization and for which we can prove satisfactory estimates.  Here $v_R(\psi)$ is no longer necessarily a linear or even a subadditive average of the $\xi_k$ and, hence,  it is not possible, as far as we can tell, to apply any version of the ergodic/subadditive ergodic theorem.  
\medskip

We are, however,  able to prove that $v_R(\psi)$ is a Lipschitz function of the variables $\xi_k$ which, as we show,  does not put too much weight on any of them in the sense that
$  |\partial_{\xi_k} v_R|_{\ell^{\infty}}  \to 0 \ \hbox{ as } \  R \to \infty.$
\medskip

In this case, the heuristic argument above suggests that the problem will homogenize.  We do need, however, to deal with the nonlinearity of the equation reflected in the fact that $v_R(\psi)$ is only a Lipschitz,  instead of linear, average of the $\xi_k$.
This is where we use the tools of concentration inequalities to show that the $v_R$'s  concentrate around their mean.

\medskip

{\it The discretized cell problem. }
We present here some  estimates for the solution of the cell problem in terms of the boundary data, which can be seen as the analogue  to the classical $L^p$-estimates for the Poisson kernel of the boundary data in the linear case.  
\medskip

We assume for simplify that $\nu = e_d$, but the proof will apply in general.  
\medskip

Recall that  $Q = [0,1)^{d-1}$  is  the unit cube on $\partial P_{e_d}$ and note that 
$$\partial P_{e_d} = \bigcup_{i \in \integer^{d-1}} (i+ Q).$$

Instead of using $Q$, it is possible to cover $\partial P_{e_d}$ using $rQ$ for some $r>0$ which can be chosen later based on the mixing rate function to optimize the estimates.  Since this would only change constants and not the rate of convergence,  we just take $r=1$ for simplicity.

\medskip
For  $X \in \Xi:=C(\overline{Q})^{\integer^{d-1}}$,  we consider the solution $u : P_{e_d} \times \Xi \to \real$ to 
\begin{equation}\label{eqn: discrete2}
\left\{
\begin{array}{lll}
F(D^2_yu) = 0  & \text{ in } & P_{e_d}, \\[1mm]
u(y;X) = \sum_{i \in \integer^{d-1}} X_i(y-i) \indicator_{i+Q}(y) & \text{ on } & \partial P_{e_d}.
\end{array}
\right.
\end{equation}

\medskip

We remark that, if $-F$ is convex (an analogous claim is true if $-F$ is concave),  then $u(\cdot;X)$ is convex  with respect to $X$. This follows from the classical fact (see \cite{CIL}) that, when $-F$ is convex, linear combinations of supersolutions  of \eqref{eqn: discrete2} are also supersolutions.

\medskip

In Lemma \ref{decorrelate} below, we show that, for each fixed $y \in P_{e_d}$,  by the maximum principle, $u(y;X)$ is continuous in the sup-norm with respect to each component $X_j$, and, hence 
 $u(y;\cdot) : (\Xi,\mathcal{B}(\Xi)) \to (\real, \mathcal{B}(\real))$ is a measurable mapping.  
\medskip
 
  We think of  the value of $u(Re_d;X)$ as a function of $R$ and $X$ and write
\begin{equation}\label{variable} 
f_R(X): =u( Re_d;X) .
\end{equation}\label{eqn: lp norms}
 
We prove Lipshitz estimates for $f_R$  with respect to  the various norms $\ell^p$-norms on $C(\overline{Q})^{\integer^{d-1}}$, which we define next. For $X \in C(\overline{Q})^{\integer^{d-1}}$,  the $\ell^p(\integer^{d-1};C(\overline{Q}))$ norms are given for 
$1\leq p < \infty$ and  $p = \infty$,  by 
\begin{equation}
 |X|_{\ell^p} = \left(\sum_{ i \in \integer^{d-1}} \sup_{y \in Q} |X_i(y)|^p \right)^{1/p}, 
 \end{equation}
 and
 $$ |X|_{\ell^\infty} = \sup_{i \in \integer^{d-1}} \sup_{y \in Q} |X_i(y)|.$$

\medskip

\begin{lem}\label{decorrelate}
Assume \eqref{op1} and \eqref{op2}. The map $f_R : C(\overline{Q})^{\integer^{d-1}} \to \real $ has the following continuity properties
 for  $X,Y \in C(\overline{Q})^{\integer^{d-1}}$ and some universal constant $C=: C(d,\lambda,\Lambda)>0$:
\begin{equation}\label{ eqn: linfty bd}
\hskip-3.5in \text{(i)} \qquad \qquad  
| f_R(X)-f_R(Y)| \leq |X-Y|_{\ell^{\infty}}.
\end{equation}
\begin{equation}\label{eqn: l1 bd} 
\hskip-3.2in \text{(ii)} \qquad \qquad |f_R(X)-f_R(Y)| \leq C R^{-\beta}|X-Y|_{\ell^{1}}.
\end{equation}
\begin{equation}\label{eqn: l2 bd}
\hskip-3.1in \text{(iii)} \qquad \qquad |f_R(X)-f_R(Y)| \leq C R^{-\beta/2}|X-Y|_{\ell^2}.
\end{equation}
\end{lem}

\begin{proof}  
The first estimate is an immediate consequence of the definition of $f_R$ and the comparison principle.

\medskip

The proof of the $\ell^1$-continuity consists of constructing  a barrier which controls how much $u(Re_d; X)$  can change when  the value of X is altered at a single site $k \in  Z^{d-1}.$ 

\medskip

The main tool  in coming up with such barrier is the 
singular solution $\Phi$ of Theorem~\ref{thm:fundamental} with $\textup{Cone}_\gamma=P_{e_d}$, that is the solution $\Phi$ to 
\begin{equation}\label{singular}
\left\{
\begin{array}{lll}
\mathcal{P}^+_{\lambda,\Lambda}(D^2\Phi) = 0 & \hbox{ in } & P_{e_d}, \\[1mm]
\Phi = 0 & \hbox{ on } & \partial P_{e_d} \setminus \{0\},
\end{array}\right.
\end{equation} 
with  $\max_{\partial B_1 \cap P_{e_d}} \Phi = 1$, which is $\beta$ homogeneous with $\beta = \beta(d,\lambda,\Lambda)$ as in \eqref{eqn: beta}.  Let 
$$\widetilde{\phi}(x) := 2m^{-1}\Phi( x+e_d ),$$
with $m := \min\{ \Phi(x) : x=(x',1) \text{ with } x' \in Q \}>0$, and observe that, in view of the uniform ellipticity and convexity of $\mathcal{P}^+_{\lambda,\Lambda}$,  
$\widetilde{\phi} \in C^{2}(\overline P_{e_d})$
(see \cite{CC95}),  and moreover,  
$$\widetilde{\phi} \geq 1 \hbox{ on }  \partial P_{e_d} \cap Q.$$
Then, for $X,Y \in \ell^{1}(\integer^{d-1})$, let,
$$\phi(x): = \sum_{j \in \integer^{d-1}}  (\sup_Q|X_j-Y_j|)\tilde{\phi}(x-j); $$
note that  the sum, which is  only over sites where $X_j \neq Y_j$,  converges in $\overline  P_{e_d}$ since 
  $X,Y \in \ell^{1}(\integer^{d-1}; C(\overline{Q}))$ and , in view of the regularity of $\widetilde{\phi}$, $\phi \in C^{2,\alpha}(\overline P_{e_d})$.
 
 \medskip
 
 Using the definition and regularity of   $\phi$ and the subadditivity of the maximal operator $\mathcal{P}^+_{\lambda,\Lambda}$ we get  
$$ F(D^2u(y;X)+D^2\phi,y) \geq 0 \ \hbox { in } \ P_{e_d} \  \hbox{ and } \ u(y;X)+\phi(y) \geq u(y;Y) \ \hbox{ on } \partial P_{e_d}.$$

\medskip 

The comparison principle and the homogeneity of $\Phi$ yield, for some universal $C$, 
$$ u(Re_d;Y) \leq u(Re_d;X) + \phi(Re_d) \leq u(Re_d;X) + m^{-1}C\sum_{j \in \integer^{d-1}} R(|j|^2+R^2)^{-(\beta+1)/2}\sup_Q|X_j-Y_j|;$$
it then follows that 
$$|f_R(X)-f_R(Y)| \leq C \sum_{j \in \integer^{d-1}} R(|j|^2+R^2)^{-(\beta+1)/2}\sup_Q|X_j-Y_j|.$$

Finally, employing  H\"{o}lder's inequality, we find for any $p \in [1 ,\tfrac{1}{1-\frac{\beta+1}{d-1}})$ or $p \in [1, \infty]$, if $\beta+1 > d-1$,
$$ |f_R(X)-f_R(Y)| \leq R^{-\beta+ (d-1)/p'}|X-Y|_{\ell^p(\integer^{d-1})},$$
where $p'$ is the H\"{o}lder dual exponent of $p$; note that, using the Poisson kernel for the upper half space, it is possible to derive the same result  as in the linear case with $\beta = d-1$. 

\medskip

For the $\ell^2$-continuity, we consider $X \in C(\overline{Q})^{\integer^{d-1}}$ and $Z \in C(\overline{Q})^{\integer^{d-1}}$ with $|Z|_{\ell^2} < +\infty$ and,  for some $t>0$ to be chosen later,  we write 
$$Z = Z_{>t}+Z_{\leq t} \ \hbox{ where } \ (Z_{>t})_j = Z_j \indicator_{\{\sup_Q |Z_j| >t\}} \ \text{and} \ \ Z_{\leq t} = Z-Z_{>t}.$$ 

Then 
\begin{align*}
 |f_R(X+Z)-f_R(X)|& \leq |f_R(X+Z)-f_R(X+Z_{>t})| + |f_R(X+Z_{>t}) - f_R(X)| \\
 &\leq |Z_{\leq t}|_{\ell^\infty}+CR^{-\beta}|Z_{> t}|_{\ell^1} \\
 & \leq t+CR^{-\beta}t^{-1}|Z|_{\ell^2}^2, 
 \end{align*}
and, after choosing $t = R^{-\beta/2}|Z|_{\ell^2}$,
$$ |f_R(X+Z)-f_R(X)| \leq  CR^{-\beta/2}|Z|_{\ell^2}^2.$$
\end{proof}

\medskip

{\it The homogenization of the cell problem \eqref{eqn: cell}.}  We apply the estimates of Lemma \ref{decorrelate} to the random problem and establish the homogenization of the cell problem \eqref{eqn: cell}
\medskip

As we did before, in what follows,  for simplicity, we assume that $\nu = e_d$, $|\psi| \leq 1$ almost surely, and we work with the solution $v_{e_d}$ of the cell problem
\begin{equation}\label{discrete2}
\left\{
\begin{array}{lll}
F(D^2v_{e_d}) = 0  & \text{ in } & P_{e_d},\\[1mm]
v_{e_d}(\cdot,\omega) = \psi(\cdot,\omega) & \text{ on } & \partial P_{e_d}.
\end{array}
\right.
\end{equation}

\smallskip

 We now transform to the setting of the previous section to apply the concentration results Theorem~\ref{concentration1} and Theorem~\ref{concentration2}.  More precisely, we consider the Banach space $C(\overline{Q})$ with the sup-norm and the associated Borel $\sigma$-algebra $\mathcal{B}(C(\overline{Q}))$ and the measure space $(\Xi, \mathcal{B}(\Xi))$, where, as before  $\Xi: = C(\overline{Q})^{\integer^{d-1}}$, and we define the measurable map $\Psi : \Omega \to \Xi$ given, for each $j\in \integer^{d-1}$,  by
$$\Psi_j(\omega)(\cdot) :=  \psi(j+\cdot, \omega) ,$$
which induces a probability measure ${P}$ on $(\Xi,\mathcal{B}(\Xi))$, the pushforward of $\P$ by $\Psi$ or, equivalently, the law of the random variable $\Psi$. Recall that, by the definition of the pushforward  measure,  for any measurable $f : \Xi \to \real$, $\int_\Xi f(X) d{P}(X) = \int_\Omega f(\Psi) d\P$.

\medskip

It is immediate, in view of the properties of $\psi$ and $\P$, that ${P}$ is stationary with respect to the natural $\integer^{d-1}$ action on $\Xi$.  Furthermore, the random field $\Psi$ on $\integer^{d-1}$ has the $\phi$-mixing property \eqref{phi mixing} adapted to the lattice with the same rate function as $\psi$ up to a dimensional constant, and, in view of \eqref{hyp: mixing},
\begin{equation}\label{eqn: mixing Psi}
 \sum_{i \in \integer^{d-1}} \phi (|i|)^{1/2} \lesssim \rho.
 \end{equation}
 We also remark that, in view of the definition of $\Psi$,  $ u(\cdot; \Psi)=v_{e_d}(\cdot,\omega)   \ \text{on } \ \partial P_{e_d}$. Since  $\psi(\cdot,\omega)$ is continuous for each $\omega$, in view of the uniqueness of the solutions to  the boundary value problem, 
 $$ u(\cdot; \Psi)=v_{e_d}(\cdot,\omega)   \ \text{ on } \ \overline P_{e_d}.$$

\medskip

Furthermore, given the  definition of $P$  and in view of the relationship between $u(\cdot; \Psi)$ and $v_{e_d}(\cdot,\omega)$,  the concentration of $u(R\nu;\Psi)$ with respect to  ${P}$  is the same as concentration of $v_{e_d}(y,\omega)$ with respect  $\P$.  In other words any probability estimates and expectations on $u(\cdot; \Psi)$ with respect to ${P}$ are the same as probability estimates and expectations on $v_{e_d}(\cdot, \cdot)$ with respect to $\P$.

\medskip

We apply next Theorem~\ref{concentration1} or Theorem ~\ref{concentration2} depending on whether $F$ and, hence,  $u(Re_d; \cdot )$ are convex or concave.

\medskip

When $F$ is neither convex nor concave, we need to assume that the ellipticity constants $(\lambda,\Lambda)$ of $F$  are such that
\begin{equation}\label{eqn: beta cond} 
2\beta(\lambda,\Lambda) > d-1,
\end{equation}
where again $\beta$ is the exponent of the upward pointing half-space singular solution for the maximal operator.  In view of \eqref{eqn: beta},   a sufficient condition for \eqref{eqn: beta cond} is
\begin{equation}\label{eqn: condition on beta}
\frac{\lambda}{\Lambda} >\frac{1}{2}\frac{d+1}{d}.
\end{equation} 
In this case, recalling that $|\psi| \leq 1$ (and therefore $|\Psi|_{\ell^\infty} \leq 1$) almost surely, we have, from the proof of Lemma \ref{decorrelate}, the following weighted Hamming distance continuity of $f_R = u(Re_d; \cdot)$:
$$|f_R(X)-f_R(Y)| \lesssim \sum_{k \in \integer^{d-1}} R(|k|^2+R^2)^{-(\beta+1)/2} \indicator_{\{X_k \neq Y_k\}} \ \hbox{ for all  } \ X,Y \in \textup{supp}( P).$$
Since we can bound the $\ell^2$-norm of the coefficients in the above inequality   as follows 
$$\sum_{k \in \integer^{d-1}} R^2(|k|^2+R^2)^{-(\beta+1)} \lesssim R^{d-1-2\beta} \int_{\real^{d-1}} (1+|x|^2)^{-(\beta+1)} dx, $$
applying Theorem \ref{concentration1} and using \eqref{eqn: beta cond} we find
\begin{equation*}\label{averageNC}
\P(\{ \omega : |v_{e_d}(Re_d,\omega)-\E v_{e_d}(Re_d,\cdot)| \geq t \}) = {P}(\{X \in \Xi: |f_R(X)- \E_{{P}} f_R| \geq t \}) \leq C \exp\left(-cR^{2\beta-(d-1)}t^2\right),
\end{equation*}
with  $c = c'/\rho^2$ and  $c'=c'(d,\lambda,\lambda)>0$.

\medskip

When $-F$ is convex,  we use the convexity of the solutions with respect to the boundary data, the $\ell^2$-continuity of Lemma~\ref{decorrelate}  and Theorem~\ref{concentration2} to obtain
\begin{equation*}\label{averageC}
\P(\{ \omega : |v_{e_d}(Re_d,\omega)-\E v_{e_d}(Re_d, \cdot)| \geq t \}) = {P}(\{ X \in \Xi: |f_R(X)- \E_{{P}} f_R| \geq t \})  \leq C \exp\left(-cR^{\beta}t^2\right),
\end{equation*}
with,  as before, $c = c'/\rho^2$ and  $c'=c'(d,\lambda,\lambda)>0$.  The same applies in the concave case.

\medskip

For brevity we  write from now on 
  \begin{equation}
  \hat{\beta} : = \beta(\lambda,\Lambda) \ \hbox{ when } F \hbox{ is either convex or concave and, otherwise, } \ \hat{\beta}: =  2 \beta(\lambda,\Lambda)- (d-1)
  \end{equation}
assuming, of course, for the latter case \eqref{eqn: condition on beta}.

\medskip

We proceed now with the proofs of Theorem~\ref{thm: convex} and Theorem~\ref{thm: nonconvex}.

\medskip

{\it Proof of Theorem~\ref{thm: convex} and Theorem~\ref{thm: nonconvex}.}  It suffices to consider the case $\nu = e_d$ since the proof for general $\nu$ is analogous, and, for simplicity,  we write  $v = v_{e_d}$. Although $v$ is a random function, in what follows, to keep the notation simple, we sometimes omit $\omega$. 
\medskip

Let $N \in 2^{\mathbb{N}}$. The first step in the proof is to show that  the expected average at height $N$
$$
 \mu_N :=  \E u(Ne_d;\Psi) = \E v(Ne_d) = \E v(Ne_d+y') \ \hbox{ for all } \ y' \in \partial P_{e_d}
 $$
 is Cauchy. 

\medskip

For a large constant $A=A(d,\lambda,\Lambda) >> (1+\hat{\beta}/2)(d-1)$ and $k \in \integer^{d-1}$, we define the events 
\begin{equation}\label{bad_set}
 E_{k}^N = \{ \omega \in \Omega: |v(Ne_d+N^{1-\hat{\beta}/2} k,\omega)-\mu_N| \geq A^{1/2}(\log N)^{1/2}N^{-\hat{\beta}/2}\}.
 \end{equation}
\smallskip

In view of the assumed stationarity and either Theorem \ref{concentration1} or \ref{concentration2} we have 
$$ \P(E_k^N ) = \P(E_0^N ) = \P(\{ \omega : |u(Ne_d;\Psi)-\mu_N| \geq A^{1/2}(\log N)^{1/2}N^{-\hat{\beta}/2}\}) \leq C\exp(-cA\log N).$$
\medskip

It is then immediate, from a simple union bound, that, if    $E^N := \bigcup_{ |k | \leq N^{\hat{\beta}}} E_{k}^N,$ then
$$\P(E^N) \leq C N^{\hat{\beta}(d-1)} \exp(-cA\log N).$$

\medskip

We fix  some large $M \in 2^{\mathbb{N}}$ and observe that,  as long as  $A > \hat{\beta}(d-1)/c$,  if $E^{\geq M} := \bigcup_{N \geq M}E^N, $ then 
$$\P(E^{\geq M}) \leq C \sum_{N \geq M} N^{\hat{\beta}(d-1)-cA} <+\infty.$$

\medskip

It follows that, for some universal sufficiently large  $M$,  $P(E^{\geq M})<1$, which, of course, implies that $\P(\Omega \setminus E^{\geq M})>0$; note that on $\Omega \setminus E^{\geq M}$,
$$|v(Ne_d+N^{1-\hat{\beta}/2} k,\omega)-\mu_N| \lesssim (\log N)^{1/2}N^{-\hat{\beta}/2} \ \hbox{ for all }  \ N \geq M \ \hbox{ and } |k| \leq N^{\hat{\beta}}. $$
\smallskip

Then, using the oscillation decay estimates (Lemma~\ref{lem:osc_decay}), for every $N \geq M$, we have
$$ |v(y'+Ne_d, \omega) - \mu_N| \lesssim  N^{1-\hat{\beta}/2}N^{-1}+(\log N)^{1/2}N^{-\hat{\beta}/2} \lesssim (\log N)^{1/2}N^{-\hat{\beta}/2} \ \hbox{ for } \ |y'| \leq N^{1+\hat{\beta}/2}.$$
\smallskip

Moreover the localization Lemma~\ref{lem: localization} gives
$$
|v(2Ne_d,\omega)-\mu_N| \lesssim \sup_{|y'| \leq N^{1+\hat{\beta}/2}} |v(y'+Ne_d,\omega)-\mu_N| +N^{1-(1+\hat{\beta}/2)} 
 \lesssim (\log N)^{1/2}N^{-\hat{\beta}/2}.
$$ 
\smallskip

Combining the previous  two estimates we get
 $$  |\mu_{2N}-\mu_N| \lesssim (\log 2N)^{1/2}(2N)^{-\hat{\beta}/2}+|v(2Ne_d,\omega)-\mu_N| 
 \lesssim (\log N)^{1/2}N^{-\hat{\beta}/2},$$
and,   for every $N,L \geq M$ in $2^{\mathbb{N}}$,
$$|\mu_{N}-\mu_{L} | \lesssim \sum_{K \geq M} (\log K)^{1/2}K^{-\hat{\beta}/2} \lesssim (\log M)^{1/2}M^{-\hat{\beta}/2}.$$
 \smallskip
 
 Thus the sequence $(\mu_N)_{N\in 2^{\mathbb{N}}}$  is Cauchy sequence and, therefore, has a limit $\mu$.
  
 \medskip
 
Next we use the above facts to prove  \eqref{eqn: exp conv disc} for $R>1$ not necessarily dyadic and $y' \in \partial P_{e_d}.$
\medskip

Let $N \in 2^{\mathbb{N}}$ such that $N \leq R \leq 2N$.  The above arguments yield that, if $R$ and, hence, $N$ are sufficiently large depending only universal constants, then for any $y' \in \partial P_{e_d}$
 $$\P(\tau_{-y'}E^{\geq N}) = \P(E^{\geq N})<1/2.$$

\medskip

The convergence rate of $\mu_N$ to $\mu$ and the same localization argument as above yield that for $\omega \not\in \tau_{-y'} E^{\geq N}$, 
 \begin{equation}\label{eqn: gen R}
\begin{array}{l}
  |v(y'+Re_d,\omega) - \mu| = |v(Re_d,\tau_{y'}\omega) - \mu|  \\
 \quad \quad \leq  |v(Re_d,\tau_{y'}\omega) - \mu_N|+ C(\log R)^{1/2}R^{-\hat{\beta}/2} \leq C(\log R)^{1/2}R^{-\hat{\beta}/2},
 \end{array}
  \end{equation}
 since $\tau_{y'} \omega \not\in E^{\geq N}$.  On the other hand, since
 $$\P (\{ \omega : |v(y'+Re_d,\omega)-\E v(y'+Re_d)| > t \}) \leq C\exp (-cR^{\hat{\beta}}t^2),$$
 taking $t = AR^{-\hat{\beta}/2}$, where $A$ is sufficiently large universal, gives 
 $$\P (\{ \omega : |v(y'+Re_d,\omega)-\E v(y'+Re_d)| > t \}) \leq 1/2,$$
and thus,
 $$ \{ |v(y'+Re_d,\omega)-\E v(y'+Re_d)| \leq  A(\log R)^{1/2}R^{-\hat{\beta}/2}\} \cap (\tau_{-y'}E^{\geq N})^C \neq \emptyset.$$
 Evaluating \eqref{eqn: gen R} for an $\omega$ in this intersection, we get 
 $$ |\E v(y'+Re_d) - \mu| \lesssim (\log R)^{1/2}R^{-\hat{\beta}/2}.$$
 
 Arguments along the same  lines also imply that the limit $\lim_{R \to \infty} v(Re_d,\omega) = \mu$ holds pointwise on the set $\Omega_0 = \Omega \setminus \cap_{M\geq 1} E^{\geq M}$ which has probability $1$.  We do not go into the details,  since this fact will not be used in the general domain case.

\qed

\medskip

For the proof of the homogenization of the general domain Dirichlet problem we need an additional estimate which we state and prove next.  Essentially, because the concentration has exponential rate,  we can combine the oscillation decay with the concentration to get uniform estimates on a polynomially large (in $R$) subset of $P_\nu$.
\begin{lem}\label{cone:D}
Under the assumptions of Theorem~\ref{thm: convex} and Theorem~\ref{thm: nonconvex} and for $\nu \in S^{d-1}$ and $R>1$, the solution $v_\nu(\cdot,\omega)$ of the cell problem \eqref{eqn: cell} satisfies the spatially uniform concentration estimate 
\begin{equation}\label{eqn: disc cell unif est}
 \P(\{ \omega :\sup_{y \in K(R,\nu)}|v_\nu(y,\omega)-\E v_\nu(y,\cdot)| > t\}) \leq C R^{d\hat{\beta}/2} \exp(-cR^{\hat{\beta}}t^2),
 \end{equation}
where 
\begin{equation}\label{box}
K(R,\nu): = \{y \in \real^d: R/2 \leq y \cdot \nu \leq 2R,  \ |y-(y\cdot \nu) \nu| \leq 3R\}.
\end{equation}
\end{lem}
\begin{proof}
Again we take $\nu = e_d$ and write $v = v_{e_d}$; the general case follows similarly.   
\medskip

We divide $K(R,e_d)$ into $O (\e^{-d})$ disjoint cubes $\tilde Q_j$ centered at $c(\tilde Q_j)$ of size  $O (\e R).$  On  each of the $\tilde Q_j$'s,  we  use the interior Lipschitz estimates along with the fact that  $d(\tilde Q_j, \partial P_{e_d}) \sim R$ to derive that, almost surely,
$$ \sup_{y \in \tilde Q_j} |v(y,\omega) - v(c(\tilde Q_j),\omega)| \lesssim \e,$$
and, hence,  
$$\sup_{y \in K(R,e_d)}  |v(y,\omega) -\E v(y)| \leq \sup_{j} |v(c(\tilde Q_j),\omega)-\E v(c(\tilde Q_j,\cdot))| + C_0\e.$$

Then, for $t > 2C_0\e$, we have
\begin{align*}
\P(\{ \omega :\sup_{y \in K(R,e_d)}  |v(y,\omega) -\E v(y)|>t\}) &\leq \P(\{ \omega :\sup_{j} |v(c(\tilde Q_j,\omega)-\E v(c(\tilde Q_j))| + C_0\e>t\}) \\
&\leq \P(\{ \omega :\sup_{j} |v(c(\tilde Q_j),\omega)-\E v(c(\tilde Q_j))|>t/2\}) \\
&\leq C \e^{-d}\P(\{ \omega : |v(c(\tilde Q_j),\omega)-\E v(c(\tilde Q_j))|>t/2\}) \\
& \leq C \e^{-d} \exp(-cR^{\hat{\beta}}t^2).
\end{align*}
On the other hand, when $t \leq 2C_0 \e$,  
$$ \P(\{ \omega :\sup_{y \in K(R,e_d)}  |v(y,\omega) -\E v(y)|>t\}) \leq 1 \leq [C^{-1} \e^{d}\exp(cC_0R^{\hat{\beta}} \e^2) ][C\e^{-d}\exp(-cR^{\hat{\beta}}t^2)].$$

Choosing $\e^2 \sim R^{-\hat{\beta}}$ and combining the two  cases yields the claim.

  \end{proof}

\medskip

\subsection*{The continuity properties of the homogenized boundary condition}\label{sec: continuity}
We discuss the continuity properties of $\mu(\nu,F,\psi)$ and, in particular, we prove Theorem~\ref{continuity00}.  Both parts of the theorem  are used to establish the  continuity of the homogenized boundary condition and the homogenization in general domains. In particular, (i) yields  the continuity of $\overline{g}(x) = \mu(g(x,\cdot,\cdot),F,\nu)$ with respect to the large scale $x$-dependence of $g$, while   (ii) implies continuity of $\overline{g}$ with respect to changing normal directions.

\medskip

{\it Proof of Theorem~\ref{continuity00}.  }
Part (i) is a direct consequence of the comparison principle.  
\medskip

To show (ii), we fix   $\nu, \nu' \in S^{d-1}$ and, without loss of generality, we assume  that $\nu \cdot \nu' \geq 1/2$  and, after a rescaling, that $|\psi| \leq 1$.  
\medskip

We estimate  $|\mu -\mu'|$, where $ \mu = \mu(\nu,F,\psi)$ and $ \mu' = \mu(\nu',F,\psi)$, by comparing the solutions $v = v_{\nu}$ and $v' = v_{\nu'}$ of the corresponding cell problems in the intersection of the half spaces $P_{\nu}$ and $P_{\nu'}$.  Here, for simplicity, we do not display in most places the dependence of $v$ and $v'$ on $\omega$.
\medskip

Let us now show that $v(\cdot, \omega) $ and $v'(\cdot, \omega) $ are close in $B_R$,  if their respective domains are close in $B_L$ for $L\gg R$.  Note that,
$$ \sup_{y \in B_L(0) \cap \partial P_{\nu'}} y \cdot \nu \leq L|\nu'-\nu|.$$ 
and, therefore,  
$$E := \{y\in \real^d:  L|\nu'-\nu| <y\cdot \nu < L, \quad  |y-(y\cdot \nu)\nu| \leq L\} \subset P_{\nu} \cap P_{\nu'}.$$

The up to the boundary Lipschitz continuity the $v$ and $v'$ (Lemma~\ref{lem:boundary_holder}) and the (almost sure) Lipschitz continuity of $\psi$, imply that, almost surely,
$$ \sup_{ y \in \partial E \cap \{y \in \real ^d: y\cdot \nu = L|\nu'-\nu| \}}|v(y,\omega) - \psi(y,\omega)| + |v'(y,\omega) - \psi(y,\omega)| \lesssim L(1 + \|D\psi\|_\infty ) |\nu'-\nu|.$$ 

Since, in view of the comparison principle,  $|v(\cdot,\omega)|, |v'(\cdot, \omega)|\leq 1$, we use the localization result Lemma \ref{lem: localization} to get  
$$ |v'(\cdot, \omega)  - v(\cdot, \omega)| \lesssim L(1 + \|D\psi\|_\infty )|\nu'-\nu|+ \frac{R}{L} \ \hbox{ in } \ E \cap B_{2R}.$$

Using the rate of convergence of the expectations from Theorem~\ref{thm: convex} or Theorem~\ref{thm: nonconvex} and taking expectations on both sides of the above inequality we get, 
\begin{align*}
| \mu-\mu'| &\leq |\mu - \E v(R\nu)| + |\mu' - \E v'(R\nu)|+\E |v(R\nu)-v'(R\nu)| \\
&\lesssim L(1 + \|D\psi\|_\infty )|\nu'-\nu|+\frac{R}{L}+(\log R)^{1/2}R^{-\hat{\beta}/2}.
\end{align*}

Choosing  $L = (\log R)^{-1/2}R^{1+\hat{\beta}/2}$ and $R^{-1} = (1+\|\psi\|_\infty)^{1/(1+\hat{\beta})}|\nu-\nu'|^{1/(1+\hat{\beta})}$ to make all the terms in the inequality above comparable size,  we get, for $\alpha' = \frac{\hat{\beta}}{2(1+\hat{\beta})}$
$$ | \mu-\mu'| \lesssim (1+\|\psi\|_\infty)^{\alpha'}\left(\log\tfrac{1}{|\nu-\nu'|}\right)^{1/2}|\nu-\nu'|^{\alpha'}, 
$$
which is the claim.

\qed
\medskip

\subsection*{The ``lifting from the hyperplane'' cell problem.}\label{lsi section}
 We prove  here Theorem~\ref{thm: lsi cell}. In view of the discussion before its statement in the previous section, for simplicity, we assume that $T=I$ and $\nu=e_d$ and we consider the cell problem 
\begin{equation}\label{eqn: other cell 1}
\left\{
\begin{array}{lll}
F(D^2v) = 0 & \hbox{ in } &   P_{e_d}, \\[1mm]
v (\cdot ,\omega) = \psi(\cdot,\omega)  & \hbox{ on } & \partial P_{e_d }.
\end{array}\right.
\end{equation}

It follows from  \eqref{psi3} and the above simplification that
$$\psi(\cdot,\omega):=\sum_{i \in \integer^{d-1}} X_i(\omega) \indicator_{i+Q}(\cdot).$$

Let $u : P_{e_d} \times \real^{\integer^{d-1}}$ be defined as in \eqref{eqn: discrete2}. The 
comparison principle yields 
$$ v (\cdot,\omega) = u(\cdot;X(\omega)) \ \hbox{ in  } \ P_{e_d}.$$
\medskip

We note that Lemma~\ref{decorrelate} still applies  and yields a universal constant  $C$ such that, for all  $Y,Z \in \real^{\integer^{d-1}}$
$$ |u(Re_d;Y)- u(Re_d;Z)| \leq CR^{-\beta/2}|Y-Z|_{\ell^2}.$$

Finally, we recall, again because of \eqref{psi3}, that the law of $X = (X_j)_{j \in \integer^{d-1}}$ on $\real^{\integer^{d-1}}$, which we call ${P}$ in analogy with the previous subsection, has LSI and, hence, from Theorem~\ref{concentration0}, we have  the concentration estimate
\begin{align*}
\P (\{ \omega :   |u(Re_d;X(\omega)) - &\E u(Re_d;X)| > t \})  = \\
&{P}(\{ Y \in \real^{\integer^{d-1}}: |u(Re_d;Y) - \int_{\real^{\integer^{d-1}}} u(Re_d;Z) d{P}(Z)|>t\})\leq C \exp (-cR^{\beta}t^2),
\end{align*}
which is as good as the convex case.  
\medskip

{\it Proof of Theorem~\ref{thm: lsi cell}.} The proof is almost the same as Theorem~\ref{thm: convex} and Theorem~\ref{thm: nonconvex} with one minor difference.  Here, instead of $\real^{d-1}$, we only have  $\integer^{d-1}$ translation invariance of the distribution of the boundary data.
The only consequence of this is that the rate of convergence of $\E v(R e_d)$ is limited also by the interior oscillation decay.
\medskip

\section{Homogenization in General Domains}\label{general}

Since we can only expect homogenization to hold on every compact subset of the domain $U$, we will consider here the rate of convergence in a subdomain $U_{R\e}$
and we will make use  of the additional free parameter $R$.  

\medskip
 We prove next a slightly more general version of Theorem~\ref{thm: main} leaving $R$ free and then explain the choice of $R$ that leads to Theorem~\ref{thm: main}.

\begin{thm}\label{thm: main 1}
Suppose that all the assumptions of  Theorem~\ref{thm: main} hold. Then, as $\ep \to 0,$ almost surely  and for every $x \in U$, $u^\e(x,\omega) \to u(x).$  Furthermore, for any $1<R \lesssim \e^{-\frac{2}{4+3\hat{\beta}}}$,  $u^\e$ concentrates about its mean with the estimate
\begin{equation}
 \P(\{ \omega :\sup_{x \in U_{R\e}} |u^\e(x,\omega)-\E u^\e(x)|>t) \leq C\exp(-cR^{\hat{\beta}}t^2+C\log\tfrac{1}{\e}\}),\\
 \end{equation}
and  the expected value converges to $\overline u$ with the estimate
\begin{equation}\label{eqn: a est 1}
\sup_{x \in U_{R\e}}|\E u^\e(x) - \overline{u}(x)| \lesssim \e^{\frac{1}{3}}R^{\frac{2}{3}}+(\log \tfrac{1}{\e})^{1/2}R^{-\hat{\beta}/2}+(\log \tfrac{1}{\e})^{1/2}(\e R)^{\frac{\hat{\beta}}{2(1+\hat{\beta})}},
\end{equation}
with  constants that depend on $\lambda,\Lambda,d,\rho,$ and the upper bound for $R\e^{\frac{2}{4+3\hat{\beta}}}$. In particular, choosing $R = \e^{-\frac{2}{4+3\hat{\beta}}}$ and $t = (\log \frac{1}{\e})^{1/2}\e^{\frac{1}{4+3\hat{\beta}}}$ yields Theorem~\ref{thm: main}.
\end{thm}

\medskip

If we assigned boundary data by ``lifting up from the hyperplane", then the homogenized boundary condition would be 
$$\overline{g}(x): = \tilde{\mu}(\psi,F, D\zeta(\zeta^{-1}(x))),$$
where $\tilde{\mu}$ comes from the alternate cell problem \eqref{eqn: other cell}.  The only difference in the estimates obtained would be in the continuity of the homogenized boundary condition, which corresponds to the third term in \eqref{eqn: a est 1}.

\medskip

We give first an outline of the strategy of the proof of  Theorem~\ref{thm: main 1}, which follows  from  a series of Lemmas.  We begin  by rescaling at a point $x_0 \in \partial U$ to $u^\e(x_0+\e y)$ and prove an estimate on the difference, in a large box of size $R$ around the origin,  between the rescaled solution in the general domain and the solution of the corresponding cell problem in $P_{\nu_{x_0}}.$  Then we use one of the cell problem homogenization results, that is either  Theorem~\ref{thm: convex} or Theorem~\ref{thm: nonconvex},  to prove an estimate of the concentration of $u^\e(x_0+\e y)$ about its mean and the convergence of $\E u^\e(x_0+\e y)$ to $\overline{g}(x_0)$ on a strip of size $\sim R$,  which is $\sim R$ away from $\partial P_{\nu_{x_0}}$.  Rescaling back to the $\e$ scale,  we use this estimate on a finite subset of $\partial U$,  which is ``$R\e$ dense'' in $\partial U$.  This implies that $u^\e$ concentrates about its mean on the boundary of the slightly smaller domain $U_{R\e}$.  Then the comparison principle gives the concentration in the interior of $U_{R\e}.$  Finally,  the  proof of  the almost sure pointwise convergence involves a careful use of the Borel-Cantelli Lemma.

\medskip

We  turn now to the full details.  We fix $x_0 \in \partial U$ and define the local rescaling of $u^\e$ near $x_0$ by
$$v^\e_{x_0}(y) := u^\e(x_0+\e y),$$
which solves
\begin{equation}
\left\{
\begin{array}{lll}
F(D^2v^\e_{x_0}) = 0 & \hbox{ in } &  \e^{-1}(U-x_0), \\[1mm]
v^\e_{x_0} (y,\omega) = g(x_0+\e y,y,\tau_{x_0/\e}\omega) & \hbox{ on } &  \e^{-1}\partial(U-x_0).
\end{array}\right.
\end{equation}
\smallskip

Let $v_{\nu}$, with $\nu = \nu_{x_0}$, be the solution to  \eqref{eqn: cell}  with boundary data 
$$
\psi(y,\omega):= g(x_0,y,\omega) \ \text{ on } \  \partial P_{\nu},
$$
and observe that, in view of the assumptions our on $g$,  $\psi$ satisfies  the hypotheses  of Theorem~\ref{thm: convex} or Theorem~\ref{thm: nonconvex} and Lemma~\ref{cone:D}.

\medskip

The next lemma provide an estimate for the difference between $v^\e_{x_0}$ and $v_{\nu}$; for its statement recall the definition \eqref{box} of the sets  $K(R,\nu).$
\begin{lem}
Assume the hypotheses of Theorem~\ref{thm: main}. For any $R>1$ and almost surely
$$ |v^{\e}_{x_0}(\cdot,\omega) - v_{\nu_{x_0}}(\cdot,\tau_{x_0/\e}\omega)| \leq C \e^{1/3}R^{2/3} \ \hbox{ in } \ K(R,\nu_{x_0}) \cap \e^{-1}(U-x_0).$$
\end{lem}
\begin{figure}[t]
\begin{center}
\includegraphics[scale=0.7, clip = true, trim = 0 0 0 0]{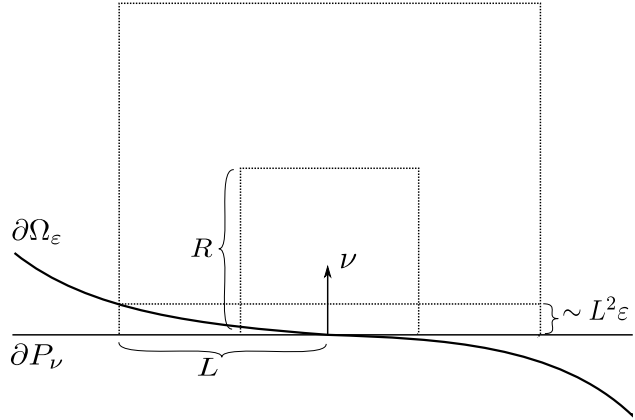}
\end{center}
\caption{Comparing with the solution of the cell problem.}
\label{fig: localize}
\end{figure}
\medskip

We remark that this is the main place where the proof of Theorem \ref{thm: main} would change,  if we assigned boundary data in the general domain by ``lifting up from the hyperplane", in which case the  estimate, with the notation of the previous section,  would instead be
$$|v^\e_{x_0}(\cdot,\omega) - v_{D\zeta (\zeta^{-1}(x_0))}(\cdot, \tau_{\e^{-1}\zeta^{-1}(x_0) \bmod \integer^{d-1}}\omega)| \leq C \e^{1/3}R^{2/3}  \ \hbox{ in } \ K(R,\nu_{x_0}) \cap \e^{-1}(U-x_0).$$

\begin{proof}
We omit reference to $x_0$,  since it is fixed for now.  
\medskip

We estimate the difference between $v^\e$ and $v_\nu$ by a localization argument in a region where $\partial U$ is close to its tangent hyperplane at $x_0$.  For a succint presentation we call $U^\e = \e^{-1}(U-x_0)$.
\medskip

It follows from \eqref{domain1} that  there exists some sufficiently large $C_0=C_0(U)$ depending on the $C^2$-regularity  of $\partial U$ such that, for any $L>R$, 
$$   \{x\in \real^d:  x \cdot \nu \geq C_0L^2 \e\} \subset U^\e \cap B_L(0) \subset \{ x\in \real^d:  x \cdot \nu \geq -C_0L^2 \e\}.$$
Note that, due to the up to the boundary Lipschitz continuity of $v^\e$, we have   
\begin{equation}\label{estimate0}
|v^\e(y,\omega) - g(x_0+\e y, y,\tau_{x_0/\e}\omega)| \leq CL^2 \e \ \hbox{ on } \ \{ x \in \real^d: x \cdot \nu = C_0L^2 \e\},
\end{equation}
and 
\begin{equation}\label{estimate01}
|v_\nu(y,\tau_{x_0/\e}\omega) - \psi( y,\tau_{x_0/\e}\omega)| \leq CL^2 \e \ \hbox{ on } \ \{ x \in \real^d:  x \cdot \nu = C_0L^2 \e\}.
\end{equation}
\medskip

Using the previous two inequalities as well as, the, uniform in  in $(y,\omega)$,  Lipshitz continuity of $g(\cdot,y,\omega)$ and the definition of $\psi$ we obtain
\begin{equation}\label{estimate02}
|v_\nu(\cdot ,\tau_{x_0/\e}\omega) - v^\e(\cdot ,\omega)| \leq CL^2 \e +C\e \lesssim L^2 \e \ \hbox{ on }  \  C_0L^2 \Pi \cap U_\e.
\end{equation}

\medskip

Fix $R>0$ with $R \ll L$ . Then using \eqref{estimate02} and the localization result (Lemma \ref{lem: localization}),  we conclude that
$$|v^\e( \cdot,\omega) - v_\nu(\cdot,\tau_{x_0/\e}\omega)| \lesssim L^2\e+\tfrac{R}{L} \ \hbox{ in } \  K(R,\nu) \cap U^\e ,$$
and choosing $L = R^{\frac{1}{3}}\e^{-\frac{1}{3}}$, so that both terms in the above estimate are the same size, we find 
\begin{equation}\label{comparison}
|v^\e(\cdot ,\omega) - v_\nu(\cdot,\tau_{x_0/\e}\omega)| \lesssim \e^{\frac{1}{3}}R^{\frac{2}{3}} \ \hbox{ in } \ K(R,\nu) \cap U^\e.
\end{equation}
\end{proof}

Next we combine \eqref{comparison} with the estimate for the cell problem given in either  Theorem~\ref{thm: convex} or Theorem~\ref{thm: nonconvex} depending on our assumptions.  

\medskip

From the previous estimates we have,   for $y \in K(R,\nu_{x_0}) \cap \e^{-1}(U-x_0)$,
\begin{equation}\label{eqn: pt var est}
|v^\e(y, \omega)-\E v^\e(y)| \leq  C_1\e^{\frac{1}{3}}R^{\frac{2}{3}}+|v_{\nu}(y, \tau_{x_0/\e}\omega) - \E v_{\nu}(y) | 
\end{equation}
and
\begin{equation}\label{eqn: pt exp est}
|\E v^\e(y) - \overline{g}(x_0)| \leq |\E v^\e(y) - \E v_{\nu}(y)|+|\E v_{\nu}(y)- \overline{g}(x_0)| \lesssim \e^{\frac{1}{3}}R^{\frac{2}{3}}+(\log R)^{1/2}R^{-\hat{\beta}/2}.
\end{equation}
Using  \eqref{eqn: pt var est} we derive, for $t>2C_1\e^{1/3}R^{2/3}$,  the following  uniform in $K(R,\nu_{x_0})$ concentration estimate of $v^\e$ about its mean:
\begin{align*}
\P(\{ \omega :\sup_{y \in K(R,\nu_{x_0})} |v^\e(y,\omega)-\E v^\e(y)|>t\}) &\leq \P(\{ \omega :\sup_{y \in K(R,\nu_{x_0})}|v_{\nu}(y,\tau_{x_0/\e}\omega) - \E v_{\nu}(y) |>t/2\})\\ 
&\leq CR^{d\hat{\beta}/2}\exp(-cR^{\hat{\beta}}t^2).
\end{align*}
On the other hand,  if $t\leq2C_1\e^{1/3}R^{2/3}$, then,  for $R \leq M \e^{-\frac{2}{4+3\hat{\beta}}}$,  we have 
$$C\exp(-cR^{\hat{\beta}}t^2)\geq C\exp(-4cC_1^2M^{\hat{\beta}+4/3}) \geq C\exp(-4cC_1^2M^{\hat{\beta}+4/3})\P(\sup_{y \in K(R,\nu_{x_0})}|v^\e(y)-\E v^\e(y)|>t).$$
Combining the two last inequalities gives  the following Lemma.
\begin{lem}\label{lem: est near a pt}
Assume the hypotheses of Theorem~\ref{thm: main}. For every $x_0 \in \partial U$ and $1<R \leq M \e^{-\frac{2}{4+3\hat{\beta}}}$, $v^\e$  concentrates about its mean uniformly in $K(R,\nu_{x_0})$ with rate 
$$\P(\{ \omega :\sup_{y \in K(R,\nu_{x_0})}|v^\e(y,\omega)-\E v^\e(y)|>t\})\leq C(M,d,\lambda,\Lambda)R^{d\hat{\beta}/2}\exp(-cR^{\hat{\beta}}t^2), $$
and its expectation $\E v^\e$  converges, again uniformly in $K(R,\nu_{x_0})$, to $\overline g(x_0)$ with rate 
$$ \sup_{y \in K(R,\nu_{x_0})} |\E v^\e(y) - \overline{g}(x_0)| \lesssim \e^{\frac{1}{3}}R^{\frac{2}{3}}+(\log R)^{1/2}R^{-\hat{\beta}/2}.$$
\end{lem}

\medskip

The next step is   to put a finite net of points on $\partial U$ and use the concentration estimate of the cell problem along with \eqref{eqn: pt var est} and a union bound to get a concentration estimate on the entire boundary of a subdomain of $U$.  

\medskip

For each $\delta>0$, choose  a finite subset ${\Gamma}_\delta$  of $\partial U$  such that
\begin{equation}\label{gf}
\partial U \subset \cup_{x \in \Gamma_\e} B(x,\delta)  \ \hbox{ and } \ |\Gamma_\delta| \lesssim \delta^{-(d-1)}.
\end{equation}
The following lemma provides a cover of $\partial U_\delta$ consisting of sets centered at points on  ${\Gamma}_\delta$.

 \begin{lem}\label{lem: geom lem}
Assume \eqref{domain1}. For sufficiently small  $\delta>0$, depending on the $C^2$-regularity  of $U$, and ${\Gamma}_\delta$
as in \eqref{gf},  $$ \partial U_\delta \subset \cup_{x_0 \in \Gamma_\delta} [x_0+\delta K(1,\nu_{x_0})].$$
 \end{lem}
 \begin{proof}
For each  $x\in \partial U_\delta$, there exists $\bar x \in \partial U$ such that  $x-\bar x = \delta\nu_{x'}$. 
\medskip

Let $ \pi_{\Gamma_\delta} x := \argmin \{ y \in \Gamma_\delta: |y-\bar x|\},$ and observe, that in view of  \eqref{gf},    
$|\pi_{\Gamma_\delta}x-\bar x| \leq 2\delta$. 
\medskip

Then
  $$ x-\pi_{\Gamma_\delta}x = (x-\bar x)+(\bar x-\pi_{\Gamma_\delta}x) = \delta \nu_{\pi_{\Gamma_\delta}x}+ \delta(\nu_{\bar x}-\nu_{\pi_{\Gamma_\delta}x})+(\bar x-\pi_{\Gamma_\delta}x).$$
 Since, in view of \eqref{domain1},  $\nu \in C^1(\partial U ;S^{d-1}),$ for $\delta$ sufficiently small depending on the $C^2-$ regularity of $U$,
  $$[\delta(\nu_{\bar x}-\nu_{\pi_{\Gamma_\delta}x})+(\bar x-\pi_{\Gamma_\delta}x)] \cdot \nu_{\pi_{\Gamma_\delta}x} \leq C\delta^2 \ \hbox{ and } \ |\delta(\nu_{\bar x}-\nu_{\pi_{\Gamma_\delta}x})+(\bar x-\pi_{\Gamma_\delta}x)| \leq 3\delta,$$
and, hence,  always for sufficiently small $\delta$,
  $$ x \in \pi_{\Gamma_\delta}x+\delta K(1,\nu_{\pi_{\Gamma_\delta}x}).$$
 \end{proof}
 
 \medskip
 
Letting $\delta = \e R$ in  the previous lemma and using  the concentration estimate at each point of $\Gamma_\delta$, we get 
\begin{align*}
\P( \{ \omega :\sup_{x\in \partial U_\delta}|u^\e(x,\omega)-\E u^\e(x)| >t \}) &\leq \P( \{ \omega : \sup_{x_0 \in \Gamma_\delta}\sup_{y \in  K(R,\nu_{x_0})}|u^\e(x_0+\e y,\omega)-\E u^\e(x_0+\e y)| >t \}) \\
&\lesssim (R \e)^{-(d-1)}\sup_{x_0 \in \Gamma_\e}\P(\{ \omega :\sup_{y \in  K(R,\nu_{x_0})}|u^\e(x_0+\e y,\omega)-\E u^\e(x_0+\e y)| >t \})  \\
& \lesssim  (R \e)^{-(d-1)}R^{d\hat{\beta}/2}\exp(-cR^{\hat{\beta}}t^2). 
\end{align*}
Then using Lemma~\ref{lem: est near a pt} as well as the up to the boundary modulus of continuity of $\overline{u}$, which is controlled by the modulus of continuity of $\overline{g}$ from Lemma~\ref{continuity00}, we find
\begin{align*}
 \sup_{x \in \partial U_\delta} |\E u^\e(x)-\overline{u}(x)| &\leq \sup_{x \in \partial U_\delta}\left( |\E u^\e(x)-\overline{g}(\pi_{\Gamma_\delta}x)|+|\overline{g}(\pi_{\Gamma_\delta}x)-\overline{u}(x)| \right)\\
 &\leq \sup_{x_0 \in \Gamma_\delta}\sup_{y \in  K(R,\nu_{x_0})}| \E u^\e(x_0+\e y) - \overline{g}(x_0)|+C(\log \tfrac{1}{\delta})^{1/2}\delta^{\alpha'} \\
 & \lesssim \e^{\frac{1}{3}}R^{\frac{2}{3}}+(\log R)^{1/2}R^{-\hat{\beta}/2}+(\log \tfrac{1}{\e})^{1/2}(\e R)^{\alpha'};
 \end{align*}
recall that  $\pi_{\Gamma_\delta} x$ is any point $x_0 \in \Gamma_\delta$ such that $x \in x_0+\delta K(1,\nu_{x_0})$, which, in view of Lemma \ref{lem: geom lem},  is well defined on $\partial U_\delta$.

\medskip

We consider next the boundary value problem 
\begin{equation}
\left\{
\begin{array}{lll}
F(D^2\widetilde{u}^\e) = 0 & \hbox{ in } & U_\delta, \\[1mm]
\widetilde{u}^\e(x) = \E u^{\e}(x,\omega) & \hbox{ on } & \partial U_\delta.
\end{array}\right.
\end{equation}
Its solution  $\widetilde{u}^\e$ can be thought as the ``extension  by  $F$''   of $\E u^{\e}\indicator_{\partial U_\delta}$  to the interior of $U_\delta$, 
\medskip

Employing the  comparison principle in $U_\delta$ yields 
$$ \P( \{ \omega :\sup_{x\in  U_\delta}|u^\e(x,\omega)-\widetilde{u}^\e(x)| >t \})  \leq \P(\{ \omega : \sup_{x\in \partial U_\delta}|u^\e(x,\omega)-\E u^\e(x,\omega)| >t \}) \lesssim  (R \e)^{-(d-1)}R^{d\hat{\beta}/2}\exp(-cR^{\hat{\beta}}t^2), $$
and 
$$ \sup_{x\in  U_\delta} |\widetilde{u}^\e(x)-\overline{u}(x)| \lesssim \e^{\frac{1}{3}}R^{\frac{2}{3}}+(\log R)^{1/2}R^{-\hat{\beta}/2}+(\log \tfrac{1}{\e})^{1/2}(\e R)^{\alpha'}.$$
\medskip

To complete the proof, it is only necessary  to  replace in the above two estimates $\widetilde{u}^\e(x)$ by $\E u^\e(x)$.  
\medskip

For this we argue as follows: For each  $x \in U_\delta$ and  $A>1$ to be chosen, we have
\begin{align*}
|\E u ^\e(x)-\widetilde{u}^\e(x)| &\leq \E|u ^\e(x)-\widetilde{u}^\e(x)|  = \int_0^\infty \P(\{ \omega :|u ^\e(x,\omega)-\widetilde{u}^\e(x)|>t\}) dt \\
& \lesssim \int_0^\infty [(R \e)^{-(d-1)}R^{d\hat{\beta}/2}\exp(-cR^{\hat{\beta}}t^2)] \wedge 1 dt \\
& \lesssim AR^{-\hat{\beta}/2}(\log \tfrac{1}{\e})^{1/2}+\int_{AR^{-\hat{\beta}/2}(\log \frac{1}{\e})^{1/2}}^\infty (R \e)^{-(d-1)}R^{d\hat{\beta}/2}\exp(-cR^{\hat{\beta}}t^2) dt \\
& \lesssim AR^{-\hat{\beta}/2}(\log \tfrac{1}{\e})^{1/2} + \int_{A(\log \frac{1}{\e})^{1/2}}^\infty (R \e)^{-(d-1)}R^{d\hat{\beta}/2}R^{-\hat{\beta}/2}\exp(-ct^2) dt \\
& \lesssim AR^{-\hat{\beta}/2}(\log \tfrac{1}{\e})^{1/2} +(R \e)^{-(d-1)}R^{d\hat{\beta}/2}R^{-\hat{\beta}/2}\e^{cA} \\
& \lesssim R^{-\hat{\beta}/2}(\log \tfrac{1}{\e})^{1/2},
\end{align*}
where for the last inequality to hold we chose $A$ large depending only on universal constants.

\medskip

Finally we discuss the almost sure convergence.

\begin{lem}
Assume the hypotheses of Theorem~\ref{thm: main}. There exists a measurable $\Omega_0 \subset \Omega$ with $\P(\Omega_0) = 1$ such that, as $\e\to 0$,
$ u^\e(x,\omega) \to \overline{u}(x) \ \hbox{ for all $x \in U$ and all $\omega \in \Omega_0$} .$
\end{lem}
\begin{proof}
In view of the fact that we already know that $\E u^\e \to \overline{u}$, it suffices to show that $|u^\e(x,\omega) - \E u^\e| \to 0$ almost surely, a fact that typically follows from a
Borel-Cantelli-type argument.
\medskip
 
Since, however,  we can apply the latter only  along sequences $\e_n \to 0$,  we first need to measure the dependence of $u^\e$ on $\e$.  
\medskip

The assumptions on $g$ and $U$ yield a universal constant $C$ such that,  for any $\e, \e' >0$ and for all $x \in \partial U,$
$$ |g(x, \frac{x}{\e},\omega) - g(x,\frac{x}{\e'},\omega)| \leq C|x| |\tfrac{1}{\e}-\tfrac{1}{\e'}| \leq C |\tfrac{1}{\e}-\tfrac{1}{\e'}|,$$
and, hence, using the comparison principle, we find that, for all $x\in U$,
$$|u^\e(x) - u^{\e'}(x)|\leq C |\tfrac{1}{\e}-\tfrac{1}{\e'}|.$$


Let $\delta(\e) = \e^{1-2\alpha_0/\hat \beta }$.  Then, as long as $|\tfrac{1}{\e}-\tfrac{1}{\e'}| \leq  c\e^{\alpha_0}$ and $\frac{\e'}{\e} \geq (2/3)^{1/\alpha_0}$ with $c$ universal and independent of $p$, and, without loss of generality, $M_p >1$,  the estimate proved above yields 
\begin{align*} 
\P(\{ \omega : \sup_{U_\delta} |u^{\e'}(x,\omega) - \E u^{\e'}(x)| > 3M_p \e'^{\alpha_0} \}) &\leq \P(\{ \omega :\sup_{U_\delta} |u^{\e}(x,\omega) - \E u^{\e}(x)| > 2M_p \e^{\alpha_0}-C|\tfrac{1}{\e}-\tfrac{1}{\e'}| \}) \\
& \leq \P(\{ \omega : \sup_{U_\delta} |u^{\e}(x,\omega) - \E u^{\e}(x)| > M_p \e^{\alpha_0} \}) \lesssim \e^p.
\end{align*} 
Now we just choose a sequence $\e_k \to 0$ such that $\e_{k+1}^{-1} - \e_k^{-1} \leq c \e_k^{\alpha_0}$.  For example, we  take $\e_k = k^{-\gamma}$, since, as  long as $\gamma \leq \min\{ c , \frac{1}{1+\alpha_0}\}$, 
$$ \e_{k+1}^{-1} = (k+1)^{\gamma} \leq k^\gamma+\gamma k^{\gamma-1} = \e_k^{-1}+\gamma \e_k^{\frac{1}{\gamma}-1} \leq \e_k^{-1} + c \e_k^{\alpha_0} \ \text{and} \ \  k^\gamma/(k+1)^\gamma  \to 1  \ \text{ as} \  k \to \infty, $$
and thus $\e_{k+1}/\e_k >(2/3)^{1/\alpha_0}$ for $k$ large enough universal. 
\medskip

Then,  for  $p > \frac{1}{\gamma}$,  we find
$$ \sum_{k} \P(\{ \omega :\sup_{U_\delta} |u^{\e_k}(x,\omega) - \E u^{\e_k}(x)| > M_p \e_k^{\alpha_0} \}) \leq \sum_k k^{-\gamma p}< + \infty,$$
which, by the  Borel-Cantelli lemma, yields 
$$ \P(\{ \omega :\sup_{U_\delta} |u^{\e_k}(x, \omega) - \E u^{\e_k}(x)| > M_p \e_k^{\alpha_0}  \ \hbox{ for infinitely many } k\}) = 0,$$
and, by the previous estimates,
\begin{align*}
 \P(\{ \omega :\sup_{U_\delta}& \  |u^{\e}(x, \omega) - \E u^{\e}(x)| > 3M_p \e^{\alpha_0} \hbox{ for arbitrarily small } \e \}) \leq \\
 &\P(\{ \omega : \sup_{U_\delta} |u^{\e_k}(x,\omega) - \E u^{\e_k}(x)| > M_p \e_k^{\alpha_0} \ \hbox{ for infinitely many } k\}) = 0.
 \end{align*}
Note  that, in fact,  we have now proven the following stronger result:  There exists an almost surely positive  random variable $\e_0: \Omega \to \real_+$ such that 
 $$ \sup_{U_{\delta(\e)}}  |u^{\e}(x,\omega) - \E u^{\e}(x)| \leq 3M_p \e^{\alpha_0} \ \hbox{ for } \ \e \leq \e_0(\omega).$$
\end{proof}

\section{Neumann problem}\label{neumann}

Here we prove the results about the Neumann problem with random oscillatory data, that is the existence of the ergodic constant (Theorem~\ref{thm: convex N} and Theorem~\ref{thm: nonconvex N}), its continuity (Theorem~\ref{continuity:N}) and the homogenization in general domains (Theorem~\ref{thm: main N}).   The methods are very similar to those used for the Dirichlet problem, but there are a few differences.  Where the proofs parallel or the same as the Dirichlet case, we simply outline the arguments or refer to the previous sections.

\medskip

We note that throughout the section we take $\beta(\lambda,\Lambda) = \frac{\lambda}{\Lambda}(d-1)-1$ to be the homogeneity exponent of the  fundamental solution  of the Neumann problem ( \eqref{fundamental:N}). 

\medskip

\subsection*{The Discretized Cell Problem }
Similarly  to the Dirichlet cell problem, we consider the solution of the Neumann cell problem \eqref{eqn: neumann cell0} as a function of the boundary data and prove Lipschitz estimates in the appropriate norms and for  discretized data.
\medskip

For simplicity we take again $\nu=e_d$, and, for    $X\in \Xi=C(\bar{Q})^{\integer^{d-1}},$ let  $u_R(y) = u_R(y;X)$ be the solution of 
\begin{equation}\label{eqn: cut off:N}
\left\{
\begin{array}{lll}
F(D_y^2u_R) = 0  & \text{ in } &  \Pi_{e_d}^{2R},\\[1mm]
\partial_{e_d} u_R= \Sigma_{i\in\integer^{d-1}} X_i(\cdot-i) \bf 1_{i+Q} & \text{ on } & \partial P_{e_d}, \\[1mm]
u_R=0& \text{ on } & \partial P_{e_d}+2Re_d;
\end{array}
\right.
\end{equation}
notice that since, in general, \eqref{eqn: cut off:N} does not have a unique solution due to the discontinuity of the boundary data, here we choose a unique $u_R$, using an approximation procedure similar to the one in Section~$2$, so that it satisfies a comparison principle. 
\medskip

With  the $u_R$  as above, let  $f_R: \Xi \to \R$  be given by 
\begin{equation}\label{variable N} 
f_R(X): =\frac{1}{R}u_R(Re_d;X).
\end{equation}
In the next lemma, which  similar to Lemma~\ref{decorrelate}, we discuss the continuity properties of $f_R$.
\begin{lem}\label{decorrelate:N}
Assume \eqref{op1} and \eqref{op2} and  let $\beta$ as in \eqref{fundamental:N}.  Then $f_R$ has the following continuity properties for  $X,Y \in C(\overline{Q})^{\integer^{d-1}}$ and some universal constant $C=: C(d,\lambda,\Lambda)>0$:
\begin{equation}\label{ eqn: linfty bd N}
\hskip-3.5in \text{(i)} \qquad \qquad  
| f_R(X)-f_R(Y)| \leq |X-Y|_{\ell^{\infty}}.
\end{equation}
\begin{equation}\label{eqn: l1 bd N} 
\hskip-2.98in \text{(ii)} \qquad \qquad |f_R(X)-f_R(Y)| \leq C R^{-(1+\beta)}|X-Y|_{\ell^{1}}.
\end{equation}
\begin{equation}\label{eqn: l2 bd N}
\hskip-2.9in \text{(iii)} \qquad \qquad |f_R(X)-f_R(Y)| \leq C R^{-(1+\beta)/2}|X-Y|_{\ell^2}.
\end{equation}
\end{lem}
\begin{proof}
The first  estimate follows from the maximum principle, since the difference $u_R(y;X)-u_R(y,Y)$  can be bounded by the linear profiles with slopes $\pm |X-Y|_{\ell^{\infty}}$.  The third claim  follows from interpolating of (i) and (ii) as in the proof of Lemma~\ref{decorrelate}, and thus to conclude we need  to prove (ii). 

\medskip

Let
$$
m = \min[ -\partial_d\Phi(x+e_d) : x \in Q]>0
$$
where $\Phi$ is given by \eqref{fundamental:N}, and define the barrier
$$
\phi_0(x) = 2m^{-1}\Phi( x+e_d ),
$$ 
so that $-\partial_{e_d}\phi_0 \geq 1$ on   $Q$.
\medskip

Then for $X,Y\in \ell^1(\integer^{d-1}; C(\bar{Q}))$ we sum over the sites where $X_j \neq Y_j$,
$$
\phi(x) = \sum_{j \in \integer^{d-1}}  (\sup_Q|X_j-Y_j|)\phi_0(x-j).
$$

\medskip

Then, as in the proof of Lemma~\ref{decorrelate},   
 $$ 
 F(D^2u_R(\cdot;X)+D^2\phi) \geq F(D^2u_R(\cdot;Y) - \mathcal{P}^+_{\lambda,\Lambda}(D^2\phi) 
 =0\ \hbox { in } \  \Pi_{e_d}^{2R}, $$
 $$
 -\partial_{e_d}( u_R(\cdot;X)+\phi) \geq -\partial_{e_d} u_R(\cdot; Y) \ \hbox{ on } \partial P_{e_d},
 $$
and, since  $\phi>0$,  $$u_R (\cdot ;X) +\phi \geq u_R(\cdot; Y) \ \text{on} \  \partial P_{e_d} + 2R e_d.$$ 

\medskip

The comparison principle then yields   $u_R(\cdot;Y) \leq u_R(\cdot; X)+\phi$ in $\Pi_{e_d}^{2R}$, and, in particular, for some universal $C>0$,
$$
 u_R(R e_d;Y) \leq u_R(R e_d; X) + \phi(R e_d) \leq u_R(R e_d;X) + m^{-1}C\sum_{j \in \integer^{d-1}} (|j|^2+R^2)^{-\beta/2}\sup_Q|X_j-Y_j|,
 $$
and  the desired estimate for $f_R = R^{-1}u_R$ follows.
\end{proof}

\subsection*{ Solving the Cell Problem}
We use  the above  estimates to solve the Neumann cell problem. 
\medskip

As before we take $\nu = e_d$, assume  $|\psi| \leq 1$ almost surely and consider the solution $v_{e_d,R}$ to the cell problem,
\begin{equation}\label{discrete2N}
\left\{
\begin{array}{lll}
F(D^2v_{e_d,R}(\cdot, \omega)) = 0  & \text{ in } &   \Pi_{e_d}^{2R},\\[1mm]
v_{e_d,R}(\cdot, \omega) = 0 & \text{ on } & \partial P_{e_d} + 2R\nu,\\[1mm]
\partial_dv_{e_d,R}(\cdot,\omega) = \psi(\cdot,\omega) & \text{ on } & \partial P_{e_d}.
\end{array}
\right.
\end{equation}
\smallskip

We transform to the setting of the previous subsection by considering $\Psi : \Omega \to \Xi$ defined, for each $j \in \integer^{d-1}$ as an element of $C(\overline{Q})$,
$$ \Psi_j(\omega)(\cdot) := \psi(x+\cdot,\omega),$$
and using the uniqueness of solutions we can identify,
$$v_{e_d,R}(\cdot,\omega) = u(\cdot;\Psi).$$

Next we sketch the argument leading to the concentration inequalities as it was described in more detail before.  
Since the $u(Re_d; \cdot)$ is Lipschitz on $\Xi$ and, in view of \eqref{hyp: mixing},  $\Psi$ is a $\phi$-mixing random field on $\integer^{d-1}$ with 
$$ \sum_{j \in \integer^{d-1}} \phi_{\Psi}(|j|)^{1/2} \lesssim  \rho,$$
we apply Theorem~\ref{concentration1} or Theorem~\ref{concentration2} to get concentration of $R^{-1}u(Re_d;\cdot)$ and, hence,  of $R^{-1}v(Re_d,\cdot)$ as well about their means. 

\medskip
 
In the convex case, we have 
$$\P(\{ \omega : R^{-1}|v_{e_d}(Re_d,\omega)-\E v_{e_d}(Re_d)| \geq t \}) \leq C \exp\left(-cR^{\beta}t^2\right) \ \hbox{ for all } \ t >0,$$
while in the non-convex case,
$$\P(\{ \omega : R^{-1}|v_{e_d}(Re_d,\omega)-\E v_{e_d}(Re_d)| \geq t \}) \leq C \exp\left(-cR^{2\beta-(d-1)}t^2\right) \ \hbox{ for all } \ t >0.$$
Define 
\begin{equation}\label{gamma}
\hat{\beta}:=  \frac{\lambda}{\Lambda}(d-1) \hbox{ if }F \hbox{ is convex or concave and, otherwise, }\hat{\beta}:= 2(\frac{\lambda}{\Lambda}- \frac{1}{2})(d-1),
\end{equation}
and note that, if we assume $\frac{\lambda}{\Lambda} > 1/2,$  then $\hat{\beta}>0$ in the non-convex case; this corresponds to \eqref{eqn: beta cond} for the Dirichlet problem.

\medskip

Following the arguments in the Dirichlet case, we prove Theorem~\ref{thm: convex N} and Theorem~\ref{thm: nonconvex N}:

\medskip

{\it Proof of Theorem~\ref{thm: convex N} and Theorem~\ref{thm: nonconvex N}.  } For $N \in 2^{\mathbb{N}}$,  
a large universal constant $A=A(d,\lambda,\Lambda) >>1$ and  $k \in \integer^{d-1}$ we consider the events 
\begin{equation}\label{bad_set:N}
 E_{k}^N = \{ \omega \in \Omega: |N^{-1}v_N(N\nu+N^{1-\hat{\beta}/2} k,\omega)-\mu_N| \geq A^{1/2}(\log N)^{1/2}N^{-\hat{\beta}/2}\}.
 \end{equation}
It follows from the  stationarity  and either Theorem \ref{concentration1} or \ref{concentration2} that 
$$ \P(E_k^N ) = \P(E_0^N ) = \P(\{ \omega : |N^{-1}v_N(N\nu,\omega)-\mu_N| \geq A^{1/2}(\log N)^{1/2}N^{-\hat{\beta}/2}\}) \leq C\exp(-cA\log N).$$
\smallskip

If $ E^{\geq M} = \bigcup_{N \geq M}E^N,$ a simple union bound yields
$$\P(E^N) \leq C N^{3\hat{\beta}(d-1)/4} \exp(-cA\log N).$$

Let  $ E^{\geq M} = \bigcup_{N \geq M}E^N.$  It follows that, for some large $M \in 2^{\mathbb{N}}$ and  as long as $A > 3\hat{\beta}(d-1)/(4c)$,  
$$\P(E^{\geq M}) \leq C \sum_{N \geq M} N^{3\hat{\beta}(d-1)/4-cA} <+\infty. $$

\medskip

Then, for $M$ sufficiently large in a universal way, $P(\Omega \setminus E^{\geq M})>0$, and, on $\Omega \setminus E^{\geq M}$,
$$|N^{-1}v_N(N\nu+N^{1-\hat{\beta}/2} k,\omega)-\mu_N| \lesssim (\log N)^{1/2}N^{-\hat{\beta}/2} \ \hbox{ for all }  \ N \geq M \ \hbox{ and } |k| \leq N^{3\hat{\beta}/4}. $$
\smallskip

We work with $\omega \in \Omega \setminus E^{\geq M}.$ It follows from  the interior oscillation decay estimates (Lemma~\ref{lem:osc_decay:N}) that, for every $N \geq M$ and  $y' \in \partial P_\nu \cap B_{N^{1+\hat{\beta}/4}},$
 $$ |N^{-1}v_N(N\nu+y',\omega) - \mu_N| \lesssim  N^{1-\hat{\beta}/2} N^{-1}+(\log N)^{1/2}N^{-\hat{\beta}/2} \lesssim (\log N)^{1/2}N^{-\hat{\beta}/2},$$ 
while   the localization estimate in Lemma~\ref{localization:N} gives, for $0 \leq t \leq N$,
  \begin{equation}\label{eqn: est3}
    |v_N(t \nu,\omega) - \tfrac{1}{2}\mu_N(2N-t)| \lesssim (\log N)^{1/2}N^{1-\hat{\beta}/2}+NN^{2(1-(1+\hat{\beta}/4))} \lesssim (\log N)^{1/2}N^{1-\hat{\beta}/2}.  
    \end{equation}
    
\smallskip
  
Next we use again a  localization estimate to compare $v_{N}$ with $v_{2N}$ in their common domain.  To this end, notice that $2N\mu_{2N}+v_{N}$ and $v_{2N}$ have the same Neumann boundary data on $\partial P_{\nu}$, while,  since  $v_{N}(y'+2N\nu,\omega) = 0$ for $y'  \in \partial P_{\nu}$,  
  $$ |2N\mu_{2N}+v_{N}(y'+2N\nu,\omega)-v_{2N}(y'+2N\nu,\omega)| \lesssim (\log N)^{1/2}N^{1-\hat{\beta}/2} \ \text{on} \ (P_\nu + 2N\nu)  \cap \overline B_{N^{1+\hat{\beta}/4}} .$$

  \medskip
  
It then follows from Lemma~\ref{localization:N} that
\begin{align}\label{eqn: est2}
 |2N\mu_{2N}+v_{N}(N\nu,\omega)-&v_{2N}(N\nu,\omega)| \nonumber \\
 &\lesssim \sup_{|y'| \leq N^{1+\hat{\beta}/2}} | 2N\mu_{2N}+v_{N}(y'+2N\nu,\omega)-v_{2N}(y'+2N\nu,\omega)| +N^{-\hat{\beta}/2} \nonumber \\
 &\lesssim (\log N)^{1/2}N^{1-\hat{\beta}/2}.
 \end{align}

Combining the previous estimates,  using \eqref{eqn: est2} and \eqref{eqn: est3} for $v_{2N}$ to estimate the three terms below and with the choice of $\omega$ above, for every $N,L \geq M$ in $2^{\mathbb{N}}$, we get
\begin{align*}
 |\mu_{N}-\mu_{2N}| &\leq |N^{-1}v_N(N\nu,\omega) - \mu_N|+|2\mu_{2N}+N^{-1}v_N(N\nu,\omega)-N^{-1}v_{2N}(N\nu,\omega)|\\ 
  & + 2|(2N)^{-1}v_{2N}(N\nu,\omega)-\tfrac{3}{2}\mu_{2N}| \lesssim (\log N)^{1/2}N^{-\hat{\beta}/2}. 
 \end{align*}

Therefore, for every $N,L \geq M$ in $2^{\mathbb{N}}$,
$$|\mu_{N}-\mu_{L} | \lesssim \sum_{K \geq M} (\log K)^{1/2}K^{-\hat{\beta}/2} \lesssim (\log M)^{1/2}M^{-\hat{\beta}/2}.$$

It follows that  $(\mu_N)_{2^{\mathbb{N}}}$ is  a Cauchy sequence and, therefore, has a limit $\mu$.  

\medskip

The extension to an estimate of $R^{-1}\E v_R(R\nu) - \mu$, for all $R>1,$ is omitted as it is just a combination of the above arguments with the ideas from the Dirichlet case.
\qed

\medskip

For the proof for general domains, we need the following spatially uniform concentration estimate. Since the notation and proof parallel that of Lemma~\ref{cone:D} we omit the details.

\begin{lem}\label{cone}
Let $v_R$ be as given above. Then, for any $t>0,$
$$
\P(\{ \omega :\sup_{\Pi^{2R}_\nu \cap \{ |y'| \leq 3R\}}R^{-1}|v_R(y,\omega) - \E v_R(y)|>t\}) \leq CR^{d\hat{\beta}/2}\exp(-CR^{\hat{\beta}} t^2),
$$

\end{lem}

\medskip

\subsection*{The  continuity of the homogenized boundary condition }
We sketch here the proof of Theorem~\ref{continuity:N}. 

\begin{proof}
The first assertion is a direct consequence of the comparison principle.  To prove the second,  we first  assume, without any loss of generality,  that $|\psi| \leq 1$.  
\medskip

Fix  $\nu_1, \nu_2 \in S^{d-1}$ and let $v_1, v_2$  and  $\mu_1,\mu_2$ be respectively the  solutions to the corresponding  Neumann cell problems and the associated ergodic constants.

\medskip

Similarly to the proof of Lemma~\ref{continuity00}, we define
$$
E = \{y\in \real^d:  L|\nu_1-\nu_2| <y\cdot \nu_1 < 2R-L|\nu_1-\nu_2|, \quad  |y-(y\cdot \nu_1)\nu_1| \leq L\} \subset P_{\nu_1} \cap P_{\nu_2},
$$
and using  the up to the boundary  $C^{1,\alpha}-$ regularity of the $v_1$ and $v_2$ (Lemma~\ref{lem:boundary_holder:N}), we find that, for every $\omega$,
$$
 \sup_{ y \in \partial E \cap \{y\cdot \nu = L|\nu_1-\nu_2| \}}|\partial_{\nu_i}v_i(y,\omega) - \psi(y,\omega)| \lesssim L^\alpha|\nu_1-\nu_2|^\alpha
 $$ 
and
$$
|\partial_{\nu_1}v_1 - \partial_{\nu_2}v_2| \leq \sup |Dv_i||\nu_1-\nu_2|  \leq C\|\psi\|_{C^\alpha}|\nu_1-\nu_2|.
$$

Therefore, for $i=1,2$,  we have 
$$
 \sup_{ y \in \partial E \cap \{y\cdot \nu_1 = L|\nu_1-\nu_2| \}}|\partial_{\nu_1}v_i(y,\omega) - \psi(y,\omega)| \leq C(\|\psi\|_{C^\alpha}|\nu_1-\nu_2|+L^\alpha |\nu_1-\nu_2|^\alpha), 
 $$ 
and,  moreover,  since $|\psi|\leq 1$, $|v_i| \leq C|\nu_1-\nu_2|L$ on $\{y \in \real^d: y\cdot\nu_1 =2R-L|\nu_1-\nu_2|\} $ and $|v_i| \leq 2R$ in $E$.

\medskip

Finally the localization Lemma \ref{localization:N} that, for $|y'| \leq R$,
$$
\frac{1}{R}|v_1(y,\omega) - v_2(y,\omega)| \leq C(|\psi|_{C^\alpha}|\nu_1-\nu_2|+L^\alpha|\nu_1-\nu_2|^\alpha+ CR^{-1}L|\nu_1-\nu_2|+ R^2L^{-2}).
$$

From here the proof follows that of Lemma~\ref{continuity00}. Without loss we can assume that $L|\nu_1-\nu_2| \leq 1$, since otherwise $\frac{1}{R}|v_1(y,\omega) - v_2(y,\omega)| \leq 2$ is a better bound.

\medskip

Using this observation to consolidate terms and  the cell problem homogenization result  after  taking expectations on both sides, we obtain 
$$|\mu_1-\mu_2| \lesssim (1+|\psi|_{C^\alpha})L^\alpha|\nu_1-\nu_2|^\alpha+R^2L^{-2}+(\log R)^{1/2} R^{-\hat{\beta}/2}.$$

Choosing $R,L$ in terms of $|\nu_1-\nu_2|$ to optimize the bound above gives the desired result.
\end{proof}

\medskip

\subsection*{The homogenization in general domains}

The proof and statement of  Theorem~\ref{thm: main N} are a bit easier than that of Theorem~\ref{thm: main}.  In contrast to the Dirichlet case the convergence rate of $u^\e$ to $\overline{u}$ is uniform in $U$.  In particular the boundary layer, represented by the parameter $R$ in Theorem~\ref{thm: main}, does not appear in the Neumann setting.
 
 \medskip

In spite of  this difference, the proof of Theorem~\ref{thm: main N} parallels the one of  Theorem~\ref{thm: main} and consists of two main steps, namely approximating  $u^\e$ in the general domain with the solution in half-space and in a local neighborhood of the ``base points", and then using the results on the half-space solutions to get the homogenization in each neighborhood.

\medskip

We do not calculate the optimal convergence rate allowed by the method in this case.  It will be evident from what follows that a more careful analysis, as in the case of the Dirichlet problem, will give the full statement of Theorem~\ref{thm: main N} with explicit exponents. 

\medskip

The goal is to show that, if $u^\e$  and $\overline u$ are respectively   the solution of the general domain Neumann problem \ref{main00:N} and the homogenized equation \eqref{mainN} with boundary data as in the statement of Theorem~\ref{thm: main N}, then, for every $p>0$ and any $ k' < k :=\min (\frac{2}{3}\frac{\alpha}{3+\alpha}, \frac{\alpha^2}{3+\alpha}, \frac{\hat{\beta}}{6})$, with $\alpha < \alpha'(\hat{\beta},\lambda,\Lambda)$ from Lemma \ref{continuity:N}, there exists $C$, which depends on universal constants $p$ and $k'$, such that,

\begin{equation}\label{concentration}
\mathbb{P}(\sup_{x \in U\setminus K}|u^\e(x,\omega)-\overline{u}(x)|> \e^{k'}) \leq C \e^p.
\end{equation}

\medskip

{\it Proof of Theorem~\ref{thm: main N}.  } Fix $\e>0$ and $t>0$,  select  a set $\Gamma_{\e^{2/3}}\subset \partial U$ of at most  $C\e^{-2/3(d-1)}$ 
boundary points such that every $\e^{2/3}$-neighborhood of a point on $\partial U$ contains at least one point in $\Gamma_{\e^{2/3}}$, let 
\begin{equation}\label{probability_set}
\Omega^t_\e:= \cup_{x\in \Gamma_{\e^{2/3}}} E^t_\e(\nu_x),
\end{equation}

where 
$$
E^t_\e(\nu) :=\{\omega: R^{-1}\sup_{y\in \Pi_\nu^{2R} \cap \{|y'| \leq 3R\}}|v_{R,\nu}(y) - \E v_{R,\nu}(y) |>t\} \hbox{ with } R=R_\e = \e^{-1/3}.
$$

and note that, in view of Lemma~\ref{cone}, 
\begin{equation}\label{rate0}
P(\Omega^t_\e) \leq C\e^{-(2/3(d-1)+d\hat{\beta}/6)}\exp (-c\e^{-\hat{\beta}/3}t^2).
\end{equation}
Choose $t_\e= c_0\e^{k'}$, with $k'<k$  and a universal $c_0$ to be chosen small.  In particular $k' < k\leq \hat\beta /6$ and, thus,  
\begin{equation}\label{rate1}
\P(\Omega_\e)  \lesssim \e^p\hbox{ for every } p< \infty\hbox{ as }\e \to 0.
\end{equation}

\medskip

We prove \eqref{concentration} by showing that, for $\omega\in\Omega \setminus \Omega_{\e}$,  
$$\phi^- \leq u^\e(\cdot,\omega) \leq \phi^+,$$ 
where $\phi^\pm$ solve \eqref{main00:N} with the modified Neumann boundary data $\mu(g(x,\cdot),F,\nu_x) \pm c_1\e^{k'}$ for $k'<k$ given in the statement of the theorem. Here $c_1$ can be chosen universally small so that $\phi^+ - \phi^- \leq \e^{k'}$, since $\phi^+-\phi^- \leq c \e^{k'} h$, where $h$ is the solution of $\mathcal{P}^+_{\lambda,\Lambda}(D^2h) = 0$ in $U \setminus K$, $h = 0$ on $\partial K$ and Neumann data identically $1$ on $\partial U$, and, therefore, has a universal upper bound.  

\medskip

The concentration estimate \eqref{concentration} then follows since, for $ \omega \in\Omega \setminus \Omega_\e$,
$$
\sup_{x\in U\setminus K}|u^\e(x,\omega) - \overline{u}(x) |\leq \sup_{x\in U\setminus K}|\phi^+-\phi^-| \leq \e^{k'}, 
$$
or, in other words,
$$
\mathbb{P}(\{ \omega : \sup_{x\in U\setminus K}|u^\e(x,\omega) -\overline{u}(x) |> \e^{k'}\}) \leq \mathbb{P}(\Omega_\e) \lesssim \e^{p}.
$$

\medskip

Below we only prove that $u^\e \leq \phi^+$, since the proof of $u^\e \geq \phi^-$ is similar.  We argue by contradiction observing that, if not, then $m:=\max _U (u^\e-\phi^+)>0$. By the maximum principle the maximum must be attained at a boundary point  $x_\e\in\partial U$, which must belong  to the $\e^{2/3}$- neighborhood of  one of ``grid points" $x_0\in \Gamma_{\e^{2/3}}$ on $\partial U$.

\medskip

Let  $\nu = \nu_{0}$ and, for $z_0 =  x_0+\e^{2/3}\nu \hbox{ and } z':= z - (z\cdot\nu)\nu,$ 
$$
T(x):= \phi(z_0)+D\phi(z_0) \cdot (x-z_0)'.
$$
Note that, in view of Lemma~\ref{continuity:N} and the fact that $g(\cdot,y,\omega)\in C^{0,1}(\R^d)$, for any $\alpha < \alpha'(\hat{\beta})$ from Lemma~\ref{continuity:N} we have 
$$
\overline{g}(x) = \mu(g(x,\cdot),F, \nu_x) \in C^{0, \alpha}(\partial U).
$$
 Then Lemma~\ref{lem:boundary_holder:N} yields that $\phi^+\in C^{1,\alpha}(\bar{U}\setminus K)$, and,  in particular,  
\begin{equation}\label{est101}
|\phi^+(x) - T(x)| \leq C|(x-z_0) \cdot \nu|+ C|(x-z_0)'|^{1+\alpha} \quad\hbox{ in }\  U \setminus K .
\end{equation}
Let  $U^\e:= \{ |(x-x_0) \cdot\nu |\leq \e^{2/3}\} \cap U$ and fix $L>1$ to be chosen later in terms of $\e$ and consider the solution  $w_\e$  to 
$$
\left\{\begin{array}{lll}
F(D^2 w_\e) = 0 &\hbox{ in } & U^\e\cap B_{L\e^{2/3}},\\[1mm]
\partial_{\nu}w_\e(\cdot ,\omega) = g(x_0, \cdot/\e,\omega) &\hbox{ on } & \partial U \cap B_{L\e^{2/3}}(x_0),\\[1mm]
w_\e(\cdot,\omega) = 0 &\hbox{ on } &(\partial P_\nu + z_0)\cap B_{L\e^{2/3}}(x_0),\\[1mm]
w_\e(\cdot ,\omega) = 0 &\hbox{ on } & U^\e\cap \partial B_{L\e^{2/3}}(x_0).
\end{array}\right. 
$$
It then follows from \eqref{est101} and the zero Dirichlet condition for $w_\e$ that 
$$
u^\e \leq m+\phi^+ \leq m+ T + w_\e + CL^{1+\alpha}\e^{2/3(1+\alpha)} \quad\hbox{ on } \quad  (\partial P_\nu + z_0) \cap U,
$$ 
and, similarly,
$$
u^\e \leq  m+ T + w_\e + C\e^{2/3} \quad\hbox{ on } \quad  U^\e\cap \partial B_{L\e^{2/3}}(x_0).
$$ 
Choose $L$ so that $L^{1+\alpha}\e^{2/3(1+\alpha)} \leq \e^{2/3}$ holds and observe that, in view of   the continuity of $g$, 
$$ \sup_{x \in B_{L\e^{2/3}}} |g(x,x/\e,\omega) - g(x_0,x/\e,\omega)| \leq C L\e^{2/3}.$$

Arguing as in  Lemma~\ref{localization:N} we estimate the difference of $u^\e$ and $m+T(x)+w_\e(x)$ using a rotated version of the barrier
$$
\varphi(x) =  CL^{1+\alpha}\e^{2/3(1+\alpha)}+CL\e^{2/3}(\e^{2/3}-x_d)+ CL^{-2}\e^{-4/3}\e^{2/3}(|x'|^2 + (d-1)\frac{\Lambda}{\lambda}(1-x_d^2)),
$$
and we get  
\begin{equation*}
u^\e \leq  m+T+ w_\e+ C(L^{1+\alpha}\e^{2/3(1+\alpha)} + L^{-2}\e^{2/3} + L\e^{4/3})\ \hbox{ on }  \ U \cap B_{\e^{2/3}}(x_0) .
\end{equation*}
Choosing $L = \e^{- \tfrac{2\alpha}{9+3\alpha}}$ and, hence, $L\e = \e^{\frac{2}{3+\alpha}}$ gives 
\begin{equation}\label{est:N}
u^\e \leq  m+T+ w_\e + C\e^{\frac{2}{3}+k} \ \hbox{ on }  \ U \cap B_{\e^{2/3}}(x_0) ,
\end{equation}
since $k \leq \frac{2\alpha}{3+\alpha}$.  
\medskip

Next we compare a rescaled version of $w_\e$ and  the solution $v_R=v_{R,\nu}$ of the cell problem \eqref{eqn: neumann cell0} with boundary data 
$$
\psi(y,\omega) : = g(x_0, y, \tau_{x_0/\e}\omega).
$$
Since, in view of \eqref{domain1}, for $c$ sufficiently small depending on the $C^2-$regularity of $\partial U$,
$$
d(\partial U-x_0, \partial P_\nu \cap B_{c\e^{\frac{2}{3+\alpha}}}(x_0))\leq \e^{\frac{4}{3+\alpha}},
$$
it follows from the up to the boundary  H\"{o}lder regularity  of  $v_R$ and $w_\e$, Lemma~\ref{lem:boundary_holder:N}, that
\begin{equation}\label{boundary2}
|\partial_{\nu} w_\e(x_0+ \cdot,\omega) -  \partial_{\nu} v_R(\cdot ,\omega)| \leq  C\e^{\frac{\alpha^2}{3+\alpha}} \  \hbox{ on }  \  (\partial P_\nu + \e^{\frac{2}{3+\alpha}}\nu) \cap B_{c\e^{3/5}}.
\end{equation}
\smallskip

Now we rescale to
$$
\tilde{w}_\e(y, \omega):= \e^{-1}w_\e(x_1+\e y, \omega),
$$
and observe that, due to the facts that $|g|\leq 1$, $R=\e^{-1/3}$ and \eqref{boundary2}, $h:=\tilde{w}_\e - v_R$ solves
$$
\left\{\begin{array}{lll}
-\mathcal{P}^+_{\lambda,\Lambda}(D^2 h) \leq 0,&\hbox{ in } & \Pi_\nu^R \cap B_{R^{\frac{3(1+\alpha)}{3+\alpha}}},\\[1mm]
|\partial_{\nu} h |\leq C\e^{\frac{3+3\alpha}{1+\alpha}} &\hbox{ on }& \partial P_\nu  \cap B_{R^{\frac{3(1+\alpha)}{3+\alpha}}},\\[1mm]
h = 0 &\hbox{ on }& (\partial P_\nu + R\nu)\cap B_{R^{\frac{3(1+\alpha)}{3+\alpha}}},\\[1mm]
h(x) \leq 2R &\hbox{ on }& \partial B_{R^{\frac{3(1+\alpha)}{3+\alpha}}} \cap \Pi_\nu^R.
\end{array}\right.
$$

Using a rotated version of the barrier
$$
\varphi(x) =C\e^{\frac{\alpha^2}{3+\alpha}} (R-x_d) +2R^{1-\frac{3(1+\alpha)}{3+\alpha}}\left(|x'|^2-(d-1)\tfrac{\Lambda}{\lambda} ((x_d)^2-R^2)\right),
$$
we conclude that, for some $C>0$ which is  independent of $\e$ and $x_0$,
\begin{equation}\label{est001}
|\tilde{w}_\e - v_R| \leq  C R(\e^{\frac{\alpha^2}{3+\alpha}} + R^{-\frac{2\alpha}{3+\alpha}}) = CR(\e^{\frac{\alpha^2}{3+\alpha}}+\e^{\frac{2}{3}\frac{\alpha}{3+\alpha}}) \leq CR \e^{k} \ \ \hbox{ in }  \ \ U \cap P_{\nu} \cap B_R.
\end{equation}

\medskip

 Note that, since $\omega \notin \Omega_\e$, the definition of $\Omega_\e$ yields that, for $\mu = \mu(g(x_0,\cdot),F,\nu)$,
\begin{equation}\label{est002}
|R^{-1}v_R(\cdot,\omega) + \mu (\frac{y}{R}\cdot\nu - 1) | \leq c_0 \e^{k'}  \ \ \hbox{ in }  \ \ \Pi_\nu^R \cap \{ |y'| \leq 3R\}.
\end{equation}
  
Rewriting \eqref{est001} and \eqref{est002} in terms of the original variable yields that, for $\e$ sufficiently small depending on $k-k'$ and $c_0$,
\begin{equation}\label{est:N2}
\e^{-2/3}|w_\e +\mu((\cdot -x_0)\cdot\nu - \e^{2/3})| \leq 3c_0 \e^{k'} \ \ \hbox{ in } \ \ U \cap B_{\e^{2/3}}(x_0).
\end{equation}

\medskip

Finally, combining  \eqref{est:N} and \eqref{est:N2}, we obtain, again for $\e$ sufficiently small, the estimate
$$
u^\e \leq m+T-\mu((\cdot-x_1)\cdot\nu - \e^{2/3})+  4c_0\e^{2/3+k'}  \ \ \hbox{ in } \ \ U \cap B_{\e^{2/3}}.
$$ 
 Recall that $-\partial_{\nu} \phi^+ \geq \mu +c_1\e^{k'}$. Then, for sufficiently small $\e>0$ and $c_0 < c_1/8$,  for $x \in \partial U \cap B_{\e^{2/3}}(x_0)$ we have
\begin{align*}
m+\phi^+(x)+ &\geq m+T(x) -\mu(x\cdot\nu - \e^{2/3}) + c_1\e^{2/3+k'} - C\e^{2/3(1+\alpha)}  \\
&\geq m+T(x) -\mu(x\cdot\nu - \e^{2/3}) + \tfrac{1}{2}c_1\e^{2/3+k'} \\
&>u^\e(x),
\end{align*}
which is a contradiction. 

\qed

\medskip

\vspace{30pt}

\bibliographystyle{plain}
\bibliography{DirichletRandomArticles}

\begin{thebibliography}{10}

\bibitem{Arisawa03}
Mariko Arisawa.
\newblock Long time averaged reflection force and homogenization of oscillating
  {N}eumann boundary conditions.
\newblock {\em Ann. Inst. H. Poincar\'e Anal. Non Lin\'eaire}, 20(2):293--332,
  2003.

\bibitem{ASS12}
S.~N. {Armstrong}, B.~{Sirakov}, and C.~K. {Smart}.
\newblock {Singular Solutions of Fully Nonlinear Elliptic Equations and
  Applications}.
\newblock {\em Archive for Rational Mechanics and Analysis}, 205:345--394,
  August 2012.

\bibitem{AS13}
Scott~N Armstrong and Charles~K Smart.
\newblock Quantitative stochastic homogenization of elliptic equations in
  nondivergence form.
\newblock {\em arXiv preprint arXiv:1306.5340}, 2013.

\bibitem{BDLS08}
G.~Barles, F.~Da~Lio, P.-L. Lions, and P.~E. Souganidis.
\newblock Ergodic problems and periodic homogenization for fully nonlinear
  equations in half-space type domains with {N}eumann boundary conditions.
\newblock {\em Indiana Univ. Math. J.}, 57(5):2355--2375, 2008.

\bibitem{BarlesMironescu12}
G.~{Barles} and E.~{Mironescu}.
\newblock {On homogenization problems for fully nonlinear equations with
  oscillating Dirichlet boundary conditions}.
\newblock {\em ArXiv e-prints}, May 2012.

\bibitem{BPL78}
Alain Bensoussan, Jacques-Louis Lions, and George Papanicolaou.
\newblock {\em Asymptotic analysis for periodic structures}, volume~5 of {\em
  Studies in Mathematics and its Applications}.
\newblock North-Holland Publishing Co., Amsterdam, 1978.

\bibitem{C99}
Luis~A. Caffarelli.
\newblock A note on nonlinear homogenization.
\newblock {\em Communications on pure and applied mathematics}, 52(7):829--838,
  1999.

\bibitem{CC95}
Luis~A. Caffarelli and Xavier Cabr{\'e}.
\newblock {\em Fully nonlinear elliptic equations}, volume~43 of {\em American
  Mathematical Society Colloquium Publications}.
\newblock American Mathematical Society, Providence, RI, 1995.

\bibitem{CFK81}
Luis~A. Caffarelli, Eugene~B. Fabes, and Carlos~E. Kenig.
\newblock Completely singular elliptic-harmonic measures.
\newblock {\em Indiana Univ. Math. J.}, 30(6):917--924, 1981.

\bibitem{CS10}
Luis~A. Caffarelli and Panagiotis~E. Souganidis.
\newblock Rates of convergence for the homogenization of fully nonlinear
  uniformly elliptic pde in random media.
\newblock {\em Inventiones mathematicae}, 180:301--360, 2010.

\bibitem{CSW05}
Luis~A. Caffarelli, Panagiotis~E. Souganidis, and L.~Wang.
\newblock Homogenization of fully nonlinear, uniformly elliptic and parabolic
  partial differential equations in stationary ergodic media.
\newblock {\em Comm. Pure Appl. Math.}, 58(3):319--361, 2005.

\bibitem{ChoiKim12}
S.~{Choi} and I.~{Kim}.
\newblock {Homogenization for nonlinear PDEs in general domains with
  oscillatory Neumann boundary data}.
\newblock {\em \textup{to appear in } JMPA}, 2013.

\bibitem{CKL12}
Sunhi {Choi}, Inwon~C. {Kim}, and Ki-Ahm {Lee}.
\newblock Homogenization of neumann boundary data with fully nonlinear
  operator.
\newblock {\em \textup{to appear in} Analysis \& PDE}, 2012.

\bibitem{CIL}
Michael~G. Crandall, Hitoshi Ishii, and Pierre-Louis Lions.
\newblock User's guide to viscosity solutions of second order partial
  differential equations.
\newblock {\em Bull. Amer. Math. Soc. (N.S.)}, 27(1):1--67, 1992.

\bibitem{DMM86}
Gianni Dal~Maso and Luciano Modica.
\newblock Nonlinear stochastic homogenization and ergodic theory.
\newblock {\em J. Reine Angew. Math.}, 368:28--42, 1986.

\bibitem{Evans89}
Lawrence~C. Evans.
\newblock The perturbed test function method for viscosity solutions of
  nonlinear {PDE}.
\newblock {\em Proc. Roy. Soc. Edinburgh Sect. A}, 111(3-4):359--375, 1989.

\bibitem{Feldman13}
William~M Feldman.
\newblock Homogenization of the oscillating dirichlet boundary condition in
  general domains.
\newblock {\em \textup{to appear in} JMPA}, 2013.

\bibitem{GVM11}
David G{\'e}rard-Varet and Nader Masmoudi.
\newblock Homogenization in polygonal domains.
\newblock {\em J. Eur. Math. Soc. (JEMS)}, 13(5):1477--1503, 2011.

\bibitem{GVM12}
David Gerard-Varet and Nader Masmoudi.
\newblock Homogenization and boundary layers.
\newblock {\em Acta Mathematica}, 209:133--178, 2012.

\bibitem{GNO}
Antoine Gloria, Stefan Neukamm, Felix Otto, et~al.
\newblock Quantification of ergodicity in stochastic homogenization: optimal
  bounds via spectral gap on glauber dynamics.
\newblock 2013.

\bibitem{GO11}
Antoine Gloria and Felix Otto.
\newblock An optimal variance estimate in stochastic homogenization of discrete
  elliptic equations.
\newblock {\em The Annals of Probability}, 39(3):779--856, 2011.

\bibitem{hoeffding}
Wassily Hoeffding.
\newblock Probability inequalities for sums of bounded random variables.
\newblock {\em J. Amer. Statist. Assoc.}, 58:13--30, 1963.

\bibitem{JKO94}
V.~V. Jikov, S.~M. Kozlov, and O.~A. Ole{\u\i}nik.
\newblock {\em Homogenization of differential operators and integral
  functionals}.
\newblock Springer-Verlag, Berlin, 1994.
\newblock Translated from the Russian by G. A. Yosifian [G. A.
  Iosif{\cprime}yan].

\bibitem{KLS12}
Carlos~E. Kenig, Fanghua Lin, and Zhongwei Shen.
\newblock Convergence rates in {$L^2$} for elliptic homogenization problems.
\newblock {\em Arch. Ration. Mech. Anal.}, 203(3):1009--1036, 2012.

\bibitem{kozlov}
S.~M. Kozlov.
\newblock Averaging of random structures.
\newblock {\em Dokl. Akad. Nauk SSSR}, 241(5):1016--1019, 1978.

\bibitem{ledoux}
Michel Ledoux.
\newblock {\em The concentration of measure phenomenon}, volume~89 of {\em
  Mathematical Surveys and Monographs}.
\newblock American Mathematical Society, Providence, RI, 2001.

\bibitem{leoni12}
Fabiana Leoni.
\newblock Explicit subsolutions and a liouville theorem for fully nonlinear
  uniformly elliptic inequalities in halfspaces.
\newblock {\em Journal de math{\'e}matiques pures et appliqu{\'e}es},
  98(5):574--590, 2012.

\bibitem{MO}
Daniel Marahrens and Felix Otto.
\newblock On annealed elliptic green function estimates.
\newblock {\em arXiv preprint arXiv:1401.2859}, 2014.

\bibitem{Marton03}
Katalin Marton.
\newblock Measure concentration and strong mixing.
\newblock {\em Studia Scientiarum Mathematicarum Hungarica}, 40(1-2):1--2,
  2003.

\bibitem{MS06}
Emmanouil Milakis and Luis~E Silvestre.
\newblock Regularity for fully nonlinear elliptic equations with neumann
  boundary data.
\newblock {\em Communications in Partial Differential Equations},
  31(8):1227--1252, 2006.

\bibitem{PV79}
G.~C. Papanicolaou and S.~R.~S. Varadhan.
\newblock Boundary value problems with rapidly oscillating random coefficients.
\newblock In {\em Random fields, {V}ol. {I}, {II} ({E}sztergom, 1979)},
  volume~27 of {\em Colloq. Math. Soc. J\'anos Bolyai}, pages 835--873.
  North-Holland, Amsterdam-New York, 1981.

\bibitem{PV82}
George~C. Papanicolaou and S.~R.~S. Varadhan.
\newblock Diffusions with random coefficients.
\newblock In {\em Statistics and probability: essays in honor of {C}. {R}.
  {R}ao}, pages 547--552. North-Holland, Amsterdam, 1982.

\bibitem{Samson00}
Paul-Marie Samson.
\newblock Concentration of measure inequalities for markov chains and ?-mixing
  processes.
\newblock {\em The Annals of Probability}, 28(1):pp. 416--461, 2000.

\bibitem{T84}
Hiroshi {Tanaka}.
\newblock {Homogenization of diffusion processes with boundary conditions.}
\newblock {Stochastic analysis and applications, Adv. Probab. Relat. Top. 7,
  411-437 (1984).}, 1984.

\bibitem{Wu96}
Jang-Mei Wu.
\newblock Harmonic measures for elliptic operators of nondivergence form.
\newblock {\em Potential Analysis}, 5(1):45--59, 1996.

\bibitem{Y85}
V.~V. Yurinski{\u\i}.
\newblock Averaging of a diffusion in a random environment.
\newblock In {\em Limit theorems of probability theory}, volume~5 of {\em Trudy
  Inst. Mat.}, pages 75--85, 175. ``Nauka'' Sibirsk. Otdel., Novosibirsk, 1985.

\bibitem{Y88}
V.~V. Yurinski{\u\i}.
\newblock On the error of averaging of multidimensional diffusions.
\newblock {\em Teor. Veroyatnost. i Primenen.}, 33(1):14--24, 1988.

\end{thebibliography}

\end{document}